\documentclass{amsart}
\usepackage{amsmath, amsthm, amssymb}

\usepackage{hyperref} 
\usepackage{enumitem}
\usepackage{stmaryrd}
\usepackage{graphicx} 

\usepackage{tikz}
\usepackage{pgf}
\usepackage{tikzscale}
\usepackage{tikz-cd}

\usetikzlibrary{decorations.pathreplacing,decorations.markings}
\usetikzlibrary{arrows, knots, external, positioning}

\usepackage{caption} 
\usepackage[a4paper,width=150mm,top=25mm,bottom=25mm,centering]{geometry}

\newtheorem{theorem}{Theorem}[section]
\newtheorem{conjecture}[theorem]{Conjecture}
\newtheorem{lemma}[theorem]{Lemma}
\newtheorem{proposition}[theorem]{Proposition}
\newtheorem{corollary}[theorem]{Corollary}

\newtheorem{problem}[theorem]{Problem}

\theoremstyle{definition}
\newtheorem{definition}[theorem]{Definition}
\newtheorem*{remark}{Remark}
\newtheorem{example}[theorem]{Example}

\renewcommand{\Bbb}{\mathbb}

\begin{document}

\title{Algebraic Montgomery-Yang Problem and Smooth Obstructions}

\author{Woohyeok Jo}
\address{Department of Mathematical Sciences, Seoul National University, Seoul 08826,
 Republic of Korea}
\email{koko1681@snu.ac.kr}

\author{Jongil Park}
\address{Department of Mathematical Sciences and Research Institute of Mathematics, Seoul National University, Seoul 08826, Republic of Korea}
\email{jipark@snu.ac.kr}

\author{Kyungbae Park}
\address{Department of Mathematics, Kangwon National University, Kangwon 24341,
 Republic of Korea}
\email{kyungbaepark@kangwon.ac.kr}

\begin{abstract}
    Let $S$ be a rational homology complex projective plane with quotient singularities. The algebraic Montgomery-Yang problem conjectures that the number of singular points of $S$ is at most three if its smooth locus is simply-connected. In this paper, we leverage results from the study of smooth 4-manifolds, including the Donaldson diagonalization theorem and Heegaard Floer correction terms, to establish additional conditions for $S$. As a result, we eliminate the possibility of a rational homology complex projective plane of specific types with four singularities. Moreover, we identify large families encompassing infinitely many types of singularities that satisfy the orbifold BMY inequality, a key property in algebraic geometry, yet are obstructed from being a rational homology complex projective plane due to smooth conditions. Additionally, we discuss computational results related to this problem, offering new insights into the algebraic Montgomery-Yang problem.
\end{abstract}
\maketitle

\section{Introduction}

A normal projective complex surface $S$ is called a \emph{rational homology $\mathbb{CP}^2$} if the homology groups of $S$ are the same as those of $\mathbb{CP}^2$ with rational coefficients, i.e., $H_*(S;\mathbb{Q})\cong H_*(\mathbb{CP}^2;\mathbb{Q})$. For instance, for relatively prime positive integers $p_1, p_2, p_3$, the weighted projective plane $\mathbb{CP}(p_1,p_2,p_3)$ is a rational homology $\mathbb{CP}^2$ with three singularities. Concerning the number of singularities of a rational homology $\mathbb{CP}^2$, Koll\'ar conjectured the following. 
    \begin{conjecture}[Algebraic Montgomery-Yang Problem, {\cite[Conjecture 30]{Kollar-2008}}]\label{conj:AMY}
        Let $S$ be a rational homology $\mathbb{CP}^2$ with quotient singularities. If the smooth locus $S^0:= S \setminus \textup{Sing}(S)$ of $S$ is simply-connected, then $S$ has at most 3 singular points.
    \end{conjecture}

While Conjecture \ref{conj:AMY} has not been completely resolved, numerous partial results have been achieved, mostly due to the work of Hwang and Keum. Let $S$ be a rational homology $\mathbb{CP}^2$ with $\pi_1(S^0)=1$. Then:
\begin{itemize}
    \item $|\textup{Sing}(S)|\leq 4$ (see \cite{Hwang-Keum-2011-2}, for example).
    \item If $S$ has at least one non-cyclic singularity, then $|\textup{Sing}(S)|\leq 3$ (\cite[Theorem 2]{Hwang-Keum-2011-1}).
    \item If $S$ is not rational or $-K_S$ is ample, then $|\textup{Sing}(S)| \leq 3$ (\cite[Theorem 1.2]{Hwang-Keum-2013}, \cite[Theorem 1.2]{Hwang-Keum-2014}).
\end{itemize}

For convenience of notation, we will abbreviate the term \textit{rational homology $\mathbb{CP}^2$ with simply-connected smooth locus} as \textit{simply-connected rational homology $\mathbb{CP}^2$} throughout this paper. However, it should be noted that even if $S$ is simply-connected, its smooth locus $S^0$ may not be. We note that it is conjectured that any simply-connected rational homology $\mathbb{CP}^2$ with quotient singularities is rational \cite[Conjecture 29]{Kollar-2008}. Since $H_1(S^0;\mathbb{Z})=0$ implies that either $K_S$ or $-K_S$ is ample \cite[Lemma 3]{Hwang-Keum-2011-1}, the aforementioned results reduce Conjecture \ref{conj:AMY} to the following. 

\begin{conjecture}\label{conj:AMY2} Let $S$ be a rational homology $\mathbb{CP}^2$ with at most $4$ cyclic singularities. Suppose that the smooth locus $S^0$ of $S$ is simply-connected. If $S$ is rational and $K_S$ is ample, then $|\textup{Sing}(S)| \leq 3$.  
\end{conjecture}

For examples of rational homology $\mathbb{CP}^2$'s that are rational and possess ample canonical divisors, refer to \cite{Hwang-Keum-2012}. From now on, we assume that a rational homology $\mathbb{CP}^2$ is a normal projective complex surface unless otherwise stated. Let $S$ be a simply-connected rational homology $\mathbb{CP}^2$ with four cyclic singularities.
 The \emph{type of singularities} of $S$ is characterized by the following set of pairs of relatively prime integers:
\[\left\{(p_1,q_1),(p_2,q_2),(p_3,q_3),(p_4,q_4)\right\},\]
where $p_i> q_i> 0$, and the $p_i$'s are the orders of the local fundamental groups of the singularities. A key property of $S$, which was instrumental in obtaining the above results, is the following \emph{orbifold Bogomolov-Miyaoka-Yau \textup{(}BMY\textup{)} inequality} (oBMY inequality, Theorem \ref{thm:oBMY}):  
\[ 
    K_S^2\leq 3e_{\textup{orb}}(S).
\]
In particular, $e_{\textrm{orb}}(S)\geq 0$ and this implies that the 4-tuple $(p_1,p_2,p_3,p_4)$ of the orders of the local fundamental groups of the singularities must  be one of the following cases: 
\begin{enumerate}[label=Case (\arabic*), leftmargin=2cm]
    \item $(2,3,5,n)$, $n\geq 7$ with $(n,30)=1$, 
    \item $(2,3,7,n)$, $n\in \{11,13,17,19,23,25,27,29,31,37,41\}$, 
    \item $(2,3,11,13)$.
\end{enumerate}

Our goal is to examine the conjecture through the lens of smooth 4-manifold theory, based on the following observation: At each singular point of $S$, a neighborhood can be chosen that is homeomorphic to the cone over a lens space $L(p_i,q_i)$. By excising these neighborhoods from $S$, we are left with a simply-connected, compact, oriented, smooth 4-manifold with $b_2=b_2^+=1$, and its boundary comprises the disjoint union $L(p_1,q_1) \amalg \cdots \amalg L(p_4,q_4)$ of lens spaces. Further, the removal of three thickened arcs that join these boundary components transforms this manifold into a simply-connected, compact, oriented, smooth 4-manifold with $b_2=b_2^+=1$ (by the standard argument in algebraic topology), whose boundary is now the connected sum $L(p_1,q_1)\# \cdots \# L(p_4,q_4)$ of lens spaces. This leads us to propose a smooth analogue of the algebraic Montgomery-Yang problem, which is stronger than the original conjecture.

\begin{problem}\label{prob:smooth}
    Among types $\{(p_1,q_1),(p_2,q_2),(p_3,p_4),(p_4,q_4)\}$ satisfying the orbifold BMY inequality, which connected sums  $L(p_1,q_1)\#L(p_2,q_2)\#L(p_3,q_3)\#L(p_4,q_4)$ of lens spaces bound simply-connected, compact, oriented, smooth 4-manifolds with $b_2=b_2^+=1$?
\end{problem}

The $p_i$'s and the $q_i$'s completely determine both $K_S^2$ and $e_{\textrm{orb}}(S)$ (see Section \ref{subsec:computation_of_K_S^2}), making it unambiguous to state that the type $\{(p_1,q_1),\dots,(p_4,q_4)\}$ satisfies the oBMY inequality. Note also that if one can demonstrate the non-existence of such connected sums of lens spaces as posed in Problem \ref{prob:smooth}, then the algebraic Montgomery-Yang problem would be resolved. Research addressing whether a given 3-manifold can bound 4-manifolds under specific conditions has been one of active topics in the study of 4-manifolds topology. A notable contribution to this field, serving as a motivation for this paper, is a result by Lisca \cite{Lisca-2007}, which classifies all lens spaces that smoothly bound a rational homology $4$-ball. This classification supports the proof of the slice-ribbon conjecture for $2$-bridge knots.

\begin{remark} 
    A smooth $S^1$-action on an odd-dimensional sphere $S^{2k-1}$, which is free except for finitely many orbits with isotropy types $\mathbb{Z}_{p_1},\dots,\mathbb{Z}_{p_n}$ of pairwise relatively prime orders, is called a \textit{pseudo-free} $S^1$-action. The original Montgomery-Yang problem asserts that a pseudo-free $S^1$-action on $S^5$ has at most three non-free orbits. Through a specific one-to-one correspondence between pseudo-free $S^1$-actions and their orbit spaces $S^5/S^1$ (as detailed in \cite[Theorem 8]{Kollar-2008}), this problem can be reformulated as:
    \begin{conjecture}[Montgomery-Yang Problem]\label{conj:original_MY} Let $M$ be a simply-connected, compact, smooth 4-manifold with $H_2(M;\Bbb Z)=\Bbb Z$ such that $\partial M$ is a connected sum $L(p_1,q_1)\# \cdots \#L(p_n,q_n)$ of lens spaces where $p_1,\dots,p_n>1$ are pairwise relatively prime. Then $n\leq 3$.     
    \end{conjecture}   
    \noindent Thus, Problem \ref{prob:smooth} can be considered as a special case of the original Montgomery-Yang problem. 
\end{remark}

We define a singularity type $\{(p_1,q_1),\dots,(p_n,q_n)\}$ as \emph{realized by a rational homology $\mathbb{CP}^2$} if there exists a rational homology $\mathbb{CP}^2$ with $n$ cyclic singularities of this specific type. Our primary result is the exclusion of Case $(2)$ and Case $(3)$ stated below Conjecture \ref{conj:AMY2}. 

\begin{theorem}\label{thm:main-1}
    If a rational homology $\mathbb{CP}^2$ with simply-connected smooth locus has four cyclic singularities, then the orders of the local fundamental groups at these singularities are given by 
    \[(p_1, p_2, p_3, p_4)=(2,3,5, p_4).\]  
\end{theorem}

In fact, the conclusion that the singularity types in Case $(3)$ cannot be realized by a simply-connected rational homology $\mathbb{CP}^2$ follows from the consequences in \cite[Section 5]{Hwang-Keum-2013} and \cite[Theorem 1.2]{Hwang-Keum-2014}. We actually show a slightly stronger statement that the connected sum of such lens spaces corresponding to Case (3) cannot be the boundary of a simply-connected \textit{smooth} rational homology $\mathbb{CP}^2$. 

In addressing Case (2) and Case (3), we discovered that the singularity types meeting the criteria set by both algebraic geometry and smooth 4-manifolds theory are exceedingly rare. This observation underscores the significant efficacy of smooth 4-manifolds theory in addressing this problem. The practical effectiveness of these theoretical results is further demonstrated in Case (1) (where $(p_1,p_2,p_3,p_4)=(2,3,5,n)$) for smaller values of $n$. This is verified using a computer program:

\begin{theorem}\label{thm:experimental}\leavevmode 
For the singularity type $(2, 3, 5, p_4)$, we have 
\begin{enumerate}[label=\normalfont(\arabic*)]
    \item  If $p_4 < 2599$, this singularity type cannot be realized by a rational homology $\mathbb{CP}^2$ with simply-connected smooth locus. 
    \item If $p_4 < 50000$, this singularity type cannot be realized by a rational homology $\mathbb{CP}^2$ with simply-connected smooth locus, possibly except for up to $16$ cases.
\end{enumerate}    
\end{theorem}

The specifics of possible 16 exceptional cases mentioned in 
Theorem~\ref{thm:experimental} above will be detailed in Section \ref{subsec:experimental_result}. Additionally we have discovered that, while large families of singularity type $(2,3,5,n)$ may not be constrained by algebraic geometry conditions, they can indeed be obstructed through an application of smooth 4-manifolds theory.

\begin{theorem}\label{thm:infinite}
    There exist infinite families of singularity type $(2,3,5,n)$ that satisfy the oBMY inequality \textup{(}Theorem \ref{thm:oBMY} \textup{(1))} but can be obstructed from realizing a rational homology $\mathbb{CP}^2$ by the condition from Donaldson's diagonalization theorem \textup{(}Corollary \ref{cor:2.6}\textup{)}.
\end{theorem}
In Section \ref{subsec:BMY_and_Donaldson}, We will present a detailed method for identifying these infinite families of singularity types. Our results indicate that a significant portion of singularities, specifically those of the type $(2,3,5,n)$, can be excluded from a consideration in the algebraic Montgomery-Yang problem. Nevertheless, there still remain infinitely many singularity types that are not ruled out by the conditions explored in this paper. An example of one such infinite family is detailed in Section \ref{subsec:mysterious_infinite_family}.

\subsection*{Acknowledgements}
The authors thank all members of the 4-manifold topology group at Seoul National University (SNU) for their invaluable comments and insights throughout this work. Jongil Park was supported by the National Research Foundation of Korea (NRF) grant funded by the Korean government (No.2020R1A5A1016126 and No.2021R1A2C1095776). 
He is also affiliated with the Research Institute of Mathematics at SNU. Kyungbae Park was supported by NRF grant funded by the Korea government (No.2021R1A4A3033098 and No.2022R1F1A1071673).

\section{Preliminary}\label{sec:conditions}
In this section, we explore the necessary conditions for cyclic singularity types \[\{(p_1,q_1),\dots,(p_n,q_n)\}\] to be realized by a rational homology $\mathbb{CP}^2$ in both algebraic geometry and the study of topological and smooth 4-manifolds. 

In this paper, the lens space $L(p,q)$ is defined as the oriented 3-manifold obtained by $-(p/q)$ surgery along the unknot in $S^3$. It is a well-known fact that two lens spaces $L(p,q)$ and $L(p',q')$ are orientation-preserving homeomorphic if and only if $p=p'$ and $q'$ is congruent to either $q$ or its inverse modulo $p$ (i.e. $q'=q^{\pm 1}\mod p$). Consequently, the $p_i$'s are uniquely determined by the type and the $q_i$'s are well-defined up to taking multicative inverse modulo $p_i$.

Let $X(p,q)$ denote the \emph{canonical negative definite plumbed 4-manifold} bounded by $L(p,q)$ as described below. Expand $p/q$ into its uniquely determined Hirzebruch-Jung continued fraction as follows:
\[ 
\frac{p}{q}=[a_1,\dots,a_\ell]:=a_1-\frac{1}{a_2-\displaystyle\frac{1}{\cdots-\displaystyle\frac{1}{a_\ell}}} \ ~~(a_i\geq 2)
\]
Here, $X(p,q)$ refers to the plumbed 4-manifold constructed from the linear graph in Figure \ref{fig:plumbing_graph_of_X(p,q)} (cf. \cite[Exercise 5.3.9(b)]{Gompf-Stipsicz-1999}). 
\begin{figure}[!th]
\centering
\begin{tikzpicture}[scale=1.1]
\draw (-2,0) node[circle, fill, inner sep=1.2pt, black]{};
\draw (-1,0) node[circle, fill, inner sep=1.2pt, black]{};
\draw (1,0) node[circle, fill, inner sep=1.2pt, black]{};
\draw (2,0) node[circle, fill, inner sep=1.2pt, black]{};

\draw (-2,0) node[below]{$-a_1$};
\draw (-1,0) node[below]{$-a_2$};
\draw (1,0) node[below]{$-a_{\ell-1}$};
\draw (2,0) node[below]{$-a_\ell$};

\draw (0,0) node{$\cdots$};

\draw (-2,0)--(-1,0) (-1,0)--(-0.5,0) (0.5,0)--(1,0)  (1,0)--(2,0) ;
\end{tikzpicture}
\caption{The plumbing graph of $X(p,q)$.}
\label{fig:plumbing_graph_of_X(p,q)}
\end{figure}
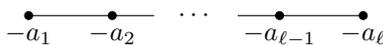

Since $L(p,p-q)$ is homeomorphic to the oriented reversal of $L(p,q)$, we deduce that $-L(p,q)$ is the boundary of the negative definite plumbed 4-manifold $X(p,p-q)$.

\begin{example} Let $p_1,p_2,p_3\geq 2$ be pairwise relatively prime integers. The \emph{weighted projective plane} 
\[ \mathbb{CP}(p_1,p_2,p_3):=(\mathbb{C}^3\setminus\{0\})/(x,y,z)\sim (\lambda^{p_1}x,\lambda^{p_2}y,\lambda^{p_3}z)~~ (\lambda \in \Bbb C\setminus\{0\})\]
is a simply-connected rational homology $\mathbb{CP}^2$ with three cyclic singularities of type \[\{(p_1,q_1), (p_2,q_2), (p_3,q_3)\},\] where the $q_i$'s are determined (up to taking the inverse modulo $p_i$) by the equation \[-q_ip_j=p_k \mod p_i\quad\text{ for } \{i,j,k\}=\{1,2,3\}.\]   
\end{example}

\subsection{Conditions from algebraic geometry}
As mentioned in the Introduction, most known results on algebraic Montgomery-Yang problem were obtained by some constraints coming from the orbifold BMY inequalities and $K_S^2$. We begin with such key ingredients used in algebraic geometry.

\subsubsection{The orbifold BMY inequalities}\label{subsec:BMY}
For a normal projective surface $S$ with quotient singularities, its \emph{orbifold Euler characteristic}, denoted as $e_{\textup{orb}}(S)$, is defined by the following formula: 
\[ 
    e_{\textup{orb}}(S):= e(S) - \sum_{p\in \textup{Sing}(S)} \left(1-\frac{1}{|G_p|} \right),
\]
where $G_p$ represents the local fundamental group at point $p$, which is the fundamental group of the link of $S$ at $p$. A key condition for $S$ is the orbifold Bogomolov-Miyaoka-Yau (BMY) inequality.

\begin{theorem}\label{thm:oBMY} 
    Let $S$ be a normal projective surface with quotient singularities. 
    \begin{enumerate}[label=\normalfont(\arabic*)]
        \item \textup{(Orbifold BMY inequality) \cite{KoNS-1989,Megyesi-1999,Miyaoka-1984,Sakai-1980}}  If the canonical class $K_S$ of $S$ is nef (numerically effective), then 
        \begin{equation*} 
            K^2_S \leq 3e_{\textup{orb}}(S).
        \label{eq:oBMY}\end{equation*}
        In particular, \[0\leq e_{\textup{orb}}(S).\]
        \item \textup{(Weak orbifold BMY inequality) \cite{KeM-1999}}  If $-K_S$ is nef, then 
        \[ 0\leq e_{\textup{orb}}(S). \]
    \end{enumerate} 
\end{theorem}

Suppose $S$ is a simply-connected rational homology $\mathbb{CP}^2$ with four cyclic singularities. In this context, the orbifold BMY inequality implies the following condition: 
    \[ 
        \sum_{p\in \textup{Sing}(S)} \left(1-\frac{1}{|G_p|}\right) \leq 3. 
    \]
Furthemore, the condition $H_1(S^0;\mathbb{Z})=0$ implies that the orders $|G_p|$ are pairwise relatively prime \cite[Lemma 3]{Hwang-Keum-2011-1}. Consequently, the 4-tuple $(p_1,p_2,p_3,p_4)$ representing the orders of the local fundamental groups must fall into one of the following categories: 
\begin{enumerate}[label=Case (\arabic*), leftmargin=2cm]
    \item $(2,3,5,n)$, $n\geq 7$ $\mathrm{with} \ (n,30)=1$, 
    \item $(2,3,7,n)$, $n\in \{11,13,17,19,23,25,27,29,31,37,41\}$, 
    \item $(2,3,11,13)$.
\end{enumerate}

\subsubsection{Computation of $K_S^2$}\label{subsec:computation_of_K_S^2} 
Let $S$ be a rational homology $\mathbb{CP}^2$ featuring four cyclic singularities of type $\{(p_1,q_1),(p_2,q_2),(p_3,q_3),(p_4,q_4)\}$. To verify the orbifold BMY inequality for this set of singularities, we perform the following computation. Removing the cone neighborhoods of the singular points yields a compact oriented smooth 4-manifold $M$ with $b_2(M)=b_2^+(M)=1$ and $\partial M=L(p_1,q_1)\amalg L(p_2,q_2)\amalg L(p_3,q_3)\amalg L(p_4,q_4)$. By attaching the canonical negative definite plumbed 4-manifolds $X(p_1,p_1-q_1) \amalg X(p_2,p_2-q_2) \amalg X(p_3,p_3-q_3)\amalg X(p_4,p_4-q_4)$ to $M$ along its boundary, we construct a closed smooth 4-manifold $W$ with $b_2^+(W)=1$. Topologically, $W$ is identical to the minimal resolution $S'$ of $S$.

Expand $p_i/(p_i-q_i)$ into the uniquely determined Hirzebruch-Jung continued fraction as follows: 
\[
    \frac{p_i}{p_i-q_i}=[n_{i,1},\dots,n_{i,\ell_i}]:=n_{i,1}-\displaystyle\frac{1}{n_{i,2}-\displaystyle\frac{1}{\cdots-\displaystyle\frac{1}{n_{i,\ell_i}}}} \ ~~ (n_{i,j}\geq 2), ~i=1,2,3,4.
\]
Then, $b_2^-(S')=b_2^-(W)=\ell_1+\cdots+\ell_4$, and $K_{S'}^2=9-(\ell_1+\cdots+\ell_4)$ by the Noether formula. Further, according to \cite[Section 3]{Hwang-Keum-2013}, we have \[ 
K_S^2=K_{S'}^2-\sum_{i=1}^4 \left( 2\ell_i -\sum_{j=1}^{\ell_i} n_{i,j} +2 - \frac{(p_i-q_i)+(p_i-q_i^{-1})+2}{p_i} \right),
\]
where $q_i^{-1}$ is the unique inverse of $q_i$ modulo $p_i$ satisfying $0<q_i^{-1}<p_i$. This leads to the formula: 
\begin{equation}\label{eq:KS_square}
K_S^2= 9 -3L +\sum_{i,j}n_{i,j} -\sum_{i=1}^4 \frac{q_i+q_i^{-1}-2}{p_i},
\end{equation}
where $L=\ell_1+\cdots+\ell_4$.

\subsubsection{Perfect squareness of $D$}
The following obstruction from algebraic geometry, which can be demonstrated by considering the minimal resolution of $S$, is another element of our analysis.

\begin{proposition}[The perfect squareness of $D$, {\cite[Lemma 3]{Hwang-Keum-2011-1}}]\label{prop:D_square} Let $S$ be a rational homology $\mathbb{CP}^2$ with four cyclic singularities of type $\{(p_1,q_1),\dots,(p_4,q_4)\}$. If $H_1(S^0;\mathbb{Z})=0$, then \[D:=p_1p_2p_3p_4K_S^2\] is a nonzero square number. \end{proposition}

In this paper, we refer to this condition as the \emph{perfect squareness of $D$} for a singularity of type 
$\{(p_1,q_1),\dots,(p_4,q_4)\}$.

\subsection{Topological Obstructions}
Let $p_1,p_2,p_3,p_4>1$ be pairwise relatively prime integers, and let $q_1,\dots,q_4$ be integers such that $0<q_i<p_i$ and $(p_i,q_i)=1$. Define the integer $q$ as $q:=q_1p_2p_3p_4+p_1q_2p_3p_4+p_1p_2q_3p_4+p_1p_2p_3q_4$. Observe that \[
    \frac{q_1}{p_1}+\frac{q_2}{p_2}+\frac{q_3}{p_3}+\frac{q_4}{p_4}=\frac{q}{p_1p_2p_3p_4}. \]
The following proposition arises from considering a relation between the intersection form of a 4-manifold and the linking form of its boundary 3-manifold.
\begin{proposition}[The linking form condition]\label{prop:linking_form} If $L(p_1,q_1)\# \cdots \# L(p_4,q_4)$ bounds a compact oriented topological 4-manifold $M$ with $H_1(M;\mathbb{Z})=0$ and $b_2(M)=b_2^+(M)=1$, then $-q$ is a quadratic residue modulo $p_1p_2p_3p_4$. \end{proposition}
It follows that if a singularity type $\{(p_1,q_1),\dots, (p_4,q_4)\}$ is realized by a rational homology $\mathbb{CP}^2$ whose smooth locus has trivial first integral homology group, then $-q$ must be a quadratic residue modulo $p_1p_2p_3p_4$. We refer to this condition as the \emph{linking form condition} for $\{(p_1,q_1),\dots, (p_4,q_4)\}$.

\begin{definition}[Linking form] Let $Y$ be an oriented 3-manifold and let $TH_1(Y;\Bbb Z)$ denote the torsion subgroup of $H_1(Y;\Bbb Z)$. For rationally nullhomologous disjoint oriented two knots $K_1, K_2$ (i.e. $[K_i]=0\in H_1(Y;\Bbb Q)$), so that $[K_i]\in TH_1(Y;\mathbb{Z})$) in $Y$, choose a rational 2-chain $c$ with $\partial c=K_1$ and define \[ L_Y([K_1],[K_2])= c \cdot K_2. \]
It can be shown that this gives a well-defined $\Bbb Q/\Bbb Z$-valued symmetric bilinear form on $TH_1(Y;\Bbb Z)$ which is nondegenerate if $Y$ is closed, see \cite[Exercise 4.5.12(c)]{Gompf-Stipsicz-1999} 
for example. 
\end{definition}

The following lemma is well-known, see \cite[Exercise 5.3.13(f),(g)]{Gompf-Stipsicz-1999} for example.

\begin{lemma}\label{lem:linking_form} Let $Y$ be a rational homology 3-sphere. \begin{enumerate}[label=\textup{(\arabic*)}]
    \item Suppose $M$ is a compact, oriented 4-manifold with $H_1(M;\Bbb Z)=0$ and $\partial M=Y$. If $A$ is any matrix for the intersection form of $M$, then $H_1(Y;\Bbb Z)$ is isomorphic to the cokernel of $A:\Bbb Z^{b_2(M)}\to \Bbb Z^{b_2(M)}$, and $(-A)^{-1}$ represents the linking form on $H_1(Y;\Bbb Z)$. 
    \item If $Y$ is obtained from $S^3$ by Dehn surgery on an oriented, framed link $L\subset S^3$ with linking matrix $B$, then $(-B)^{-1}$ represents the linking form on $H_1(Y;\Bbb Z)$ with respect to the generating set given by meridians.
\end{enumerate}
\end{lemma}

\begin{proof}[Proof of Proposition \ref{prop:linking_form}]
    Recall that the lens space $L(r,s)$ is obtained by $-(r/s)$-surgery along the unknot in $S^3$. Therefore, the linking form on $H_1(L(r,s);\Bbb Z)=\Bbb Z_r$ is represented by $\displaystyle\left(\frac{s}{r}\right)$, as indicated in Lemma \ref{lem:linking_form} (2). Consequently, for $Y=L(p_1,q_1)\# \cdots \# L(p_4,q_4)$, the linking form on $H_1(Y;\Bbb Z)=\Bbb Z_{p_1}\oplus \cdots \oplus \Bbb Z_{p_4}$ is represented by $\left(\frac{q_1}{p_1}\right)\oplus \cdots \oplus \left(\frac{q_4}{p_4}\right)$. Note that $\left(\frac{q_1}{p_1}\right)\oplus \cdots  \oplus \left(\frac{q_4}{p_4}\right)$ corresponds to $\left(\frac{q}{p_1p_2p_3p_4}\right)$ under the isomorphism $\Bbb Z_{p_1}\oplus \cdots \oplus \Bbb Z_{p_4}\cong \Bbb Z_{p_1p_2p_3p_4}$, mapping $(1,1,1,1)$ to $1$.

    On the other hand, given that $b_2(M)=b_2^+(M)=1$, the intersection form of $M$ is represented by a matrix $(p)$ for some positive integer $p$. According to Lemma \ref{lem:linking_form} (1), $H_1(Y;\Bbb Z)$ is isomorphic to $\Bbb Z_p$ and the linking form on $H_1(Y;\Bbb Z)=\Bbb Z_p$ is represented by $\left(-\frac{1}{p}\right)$. Consequently, this implies that $p=p_1p_2p_3p_4$ and that the form $\left(\frac{q}{p_1p_2p_3p_4}\right)$ is isomorphic to $\left(-\frac{1}{p}\right)$.
    Since symmetric bilinear forms on $\Bbb Z_p$ isomorphic to $\left(-\frac{1}{p}\right)$ are precisely those of the form $\left(-\frac{a^2}{p}\right)$ with $(a,p)=1$, it follows that $q\equiv-a^2$ mod $p$ for some integer $a$. 
\end{proof}

\begin{remark} Suppose that $-q$ is a quadratic residue modulo $p_1p_2p_3p_4$, and hence the linking form of $Y=L(p_1,q_1)\#\cdots \# L(p_4,q_4)$ can be represented by $\left(-\frac{1}{p_1p_2p_3p_4}\right)$. Then it is known that $Y$ bounds a compact oriented \emph{topological} 4-manifold $M$ with $H_1(M;\Bbb Z)=0$ and $b_2(M)=b_2^+(M)=1$ \cite{Boyer-1986}. 
Therefore, the converse of Proposition \ref{prop:linking_form} is also valid. 
\end{remark}

The corollary below follows directly from the linking form condition, which imposes a strict restriction on the order of the fourth singularity modulo $30$ for a singularity of type $(2,3,5,n)$.

\begin{corollary}
    Suppose that a singularity of type $\{(2,1),(3,2),(5,q_3),(p_4,q_4)\}$ with $(p_4,30)=1$ is realized by a rational homology $\mathbb{CP}^2$ with simply-connected smooth locus. If  $q_3=1$ or $4$, then \[p_4\equiv 1 \text{ or } 19\mod 30.\]
    If $q_3=2$, then \[p_4\equiv 7\text{ or } 13\mod 30.\]
\end{corollary}

\begin{proof} Suppose that a singularity of type $\{(2,1),(3,2),(5,1),(p_4,q_4)\}$ with $(p_4,30)=1$ can be realized by a simply-connected rational homology $\mathbb{CP}^2$. In this case, we have $-q=-(41p_4+30q_4)\equiv 19p_4-30q_4\mod 30p_4$. When reduced modulo $3$, we deduce that $p_4$ must be a quadratic residue modulo $3$. Therefore, $p_4\equiv 1\mod 3$. Similarly, reducing modulo $5$ yields $p_4\equiv 1\text{ or } 4\mod 5$. Consequently, it follows that $p_4\equiv 1\text{ or }19\mod 30$.

Applying the same argument, if a singularity of type $\{(2,1),(3,2),(5,2),(p_4,q_4)\}$ is realized by a simply-connected rational homology $\mathbb{CP}^2$, then $p_4 \equiv 7 \text{ or } 13 \mod 30$. Similarly, for a singularity of type $\{(2,1),(3,2),(5,4),(p_4,q_4)\}$, we have $p_4 \equiv 1 \text{ or } 19 \mod 30$.
\end{proof}

\subsubsection{The linking form condition and the lattice-theoretic argument for the minimal resolution}

In this subsection, we investigate a relation between the linking form condition (Proposition \ref{prop:linking_form}) and the lattice-theoretic argument presented in \cite[Section 6]{Hwang-Keum-2011-2} (also noted in \cite[Remark 5.1]{Hwang-Keum-2013}). Interestingly, we find that the linking form condition is actually more stringent than the lattice argument in \cite[Section 6]{Hwang-Keum-2011-2}. Please note that the content of this subsection will not be referenced further in the remainder of this paper.

Suppose $S$ is a rational homology $\mathbb{CP}^2$ with four cyclic singularities of type $\{(p_1,q_1), \dots, (p_4,q_4)\}$. Let $W$ be the closed oriented 4-manifold constructed in Section \ref{subsec:computation_of_K_S^2}, which is homeomorphic to the minimal resolution $S'$ of $S$. Then the intersection form $Q_W$ of $W$ is unimodular and indefinite. We have $b_2^+(W)=1$ and $b_2^-(W)=L$, where $L=\ell_1+\cdots+\ell_4$ is as in Section \ref{subsec:BMY}. (Note that $L$ is determined by the $p_i$'s and the $q_i$'s.) According to the algebraic classification of unimodular indefinite lattices (see \cite{Milnor_Husemoller}, for example), the signature of an even unimodular lattice is divisible by 8, and \[
Q_W \cong \begin{cases}
    \langle 1\rangle \oplus L \langle -1\rangle, & \textrm{if $Q_W$ is odd}, \\
    H\oplus \frac{L-1}{8} E_8, &\textrm{if $Q_W$ is even},
\end{cases}
\]
where $H$ and $E_8$ are the unimodular lattices represented by the matrices \[
\begin{bmatrix}
     0& 1 \\ 1 &0 
\end{bmatrix}~~\textrm{and}~~\begin{bmatrix}
    -2 & 1 & 0&0 & 0& 0& 0&0 \\
    1 & -2 & 1 & 0& 0&0 &0 &0 \\
   0  & 1 & -2 & 1 &0 & 0& 0&0 \\
    0 & 0& 1 & -2 & 1 & 0& 0& 0\\
     0& 0&0 & 1 & -2 & 1 & 0& 1 \\
     0& 0&0 &0 & 1 & -2 & 1 & 0\\
    0 & 0&0 & 0& 0& 1 & -2 & 0\\
    0 & 0& 0& 0& 1 &0 & 0& -2    
\end{bmatrix},
\]
respectively.

Now, consider a submanifold $X:=X(p_1,p_1-q_1) \amalg \cdots \amalg X(p_4,p_4-q_4)$ in $W$. Since $H_2(L(p_i,q_i);\mathbb{Z})$ is trivial for all $i$, the intersection lattice of $X$, $Q_X\cong\bigoplus_{i=1}^4 Q_{X(p_i,p_i-q_i)}$, \emph{embeds into} $Q_W$, i.e. there exists an injective map from $Q_X$ to $Q_W$ that preserves the bilinear form. In summary, we have the following: 

\begin{proposition}\label{prop:lattice-theoretic} 
If a singularity of type $\{(p_1,q_1),\dots, (p_4,q_4)\}$ is realized by a rational homology $\mathbb{CP}^2$, then the direct sum 
\[
    \bigoplus_{i=1}^4 Q_{X(p_i,p_i-q_i)}
\]
of linear lattices embeds into either $\langle 1\rangle \oplus L \langle -1\rangle$ or $H\oplus \frac{L-1}{8} E_8$.
\end{proposition}

Thus, this can be used as an obstruction of realizing a rational homology $\mathbb{CP}^2$ with a given singularity type. The Local-Global principle and $\epsilon$-invariants can be used to check whether there is an embedding of Proposition \ref{prop:lattice-theoretic}(see \cite[Section 6]{Hwang-Keum-2011-2}).

On the other hand, if $-q$ (as defined above Proposition \ref{prop:linking_form}) is a quadratic residue modulo $p_1p_2p_3p_4$, then $L(p_1,q_1)\# \cdots \# L(p_4,q_4)$ bounds a compact oriented topological 4-manifold $M$ with $b_2(M)=b_2^+(M)=1$ as noted in the remark following the proof of Proposition \ref{prop:linking_form}. Then by attaching $\natural_{i=1}^4 X(p_i,p_i-q_i)$ to $M$, we can construct a closed 4-manifold $W'$, whose intersection form is unimodular, indefinite, and has the same $b_2^+$ and $b_2^-$ as those of $W$. Consequently, the lattice $\bigoplus_{i=1}^4 Q_{X(p_i,p_i-q_i)}$ embeds into $Q_{W'}$. This concludes that the obstruction from Proposition~\ref{prop:lattice-theoretic} above is less restrictive than the linking form condition.
\begin{proposition} If $-q$ is a quadratic residue modulo $p_1p_2p_3p_4$, then the direct sum \[
\bigoplus_{i=1}^4 Q_{X(p_i,p_i-q_i)}
\]
of linear lattices embeds into either $\langle 1\rangle \oplus L \langle -1\rangle$ or $H\oplus \frac{L-1}{8} E_8$.    
\end{proposition}

\subsection{Conditions from the theory of smooth 4-manifolds}
As highlighted in the Introduction, the algebraic Montgomery-Yang problem can be reformulated into a question concerning smooth 4-manifolds. In this subsection, we recall some conditions from smooth 4-manifolds theory that are applicable to this problem.

\subsubsection{Donaldson's diagonalization theorem}\label{subsec:Donaldson}
By Freedman \cite{Freedman-1982}, any symmetric unimodular bilinear form over integers can be the intersection form of a (simply-connected) closed, oriented topological 4-manifold. However, Donaldson's diagonalization theorem states that there is a significant constraint on the intersection form of closed, oriented, \emph{definite, smooth} 4-manifolds. We can utilize this constraint to provide a condition for a smooth 4-manifold whose boundary is a specified 3-manifold. The argument presented here is inspired by Lisca's work \cite{Lisca-2007}.

For a positive integer $n$, let $\{e_1,\dots,e_n\}$ be the standard basis for $\Bbb Z^n$. We denote by $-\Bbb Z^n$ the negative definite lattice $(\Bbb Z^n, \langle\cdot,\cdot\rangle)$ given by $\langle e_i,e_j\rangle = -\delta_{i,j}$, with $\delta_{i,j}$ being the Kronecker delta. Additionally, for a compact, oriented 4-manifold $X$ and its intersection form \[Q_X:H_2(X;\Bbb Z)/\textrm{Tor}\times H_2(X;\Bbb Z)/\textrm{Tor}\to \Bbb Z,\] we will simply denote the lattice $(H_2(X;\Bbb Z)/\textrm{Tor}, Q_X)$ by $Q_X$.

\begin{theorem}[Donaldson's diagonalization theorem, {\cite{Donaldson-1983,Donaldson-1987}}]\label{thm:Donaldson} If the intersection form $Q_W$ of a closed, oriented, smooth 4-manifold $W$ is negative definite, then $Q_W$ is isomorphic to $-\Bbb Z^n$, where $n=b_2(W)=b_2^-(W)$. 
\end{theorem}

Now, suppose that $L(p_1,q_1)\# \cdots \# L(p_4,q_4)$ bounds an oriented smooth 4-manifold $M$ with $b_2(M)=b_2^+(M)=1$. Consider the boundary sum $X:=X(p_1,q_1)\natural \cdots \natural X(p_4,q_4)$ of the canonical negative definite plumbed 4-manifolds of the $L(p_i,q_i)$'s, which is a negative definite 4-manifold with $\partial X=L(p_1,q_1)\# \cdots \# L(p_4,q_4)$ and $b_2(X)=n:=\sum_{i=1}^4 b_2(X(p_i,q_i))$. By attaching $X$ to $-M$ along the boundary, we obtain a closed negative definite smooth 4-manifold $W: = X \cup_{\partial} (-M)$. Then, by applying Donaldson's diagonalization theorem to $W$, we find that $Q_W$ is isomorphic to $-\mathbb{Z}^{n+1}$, leading to the following corollary.




\begin{corollary}[The condition from Donaldson's diagonalization theorem]\label{cor:2.6} 
    If the connected sum $L(p_1,q_1)\# \cdots \# L(p_4,q_4)$ of lens spaces bounds a compact, oriented, smooth 4-manifold $M$ with $b_2(M)=b_2^+(M)=1$, then $\bigoplus_{i=1}^4 Q_{X(p_i,q_i)}$ embeds into $-\mathbb{Z}^{n+1}$, where $n=\sum_{i=1}^4 b_2(X(p_i,q_i))$. 
\end{corollary}

Therefore, if we show that it is impossible to find an embedding from $\bigoplus_{i=1}^4 Q_{X(p_i,q_i)}$ into $-\mathbb{Z}^{n+1}$, 
we can conclude that a singularity of type $\{(p_1,q_1),\dots,(p_4,q_4)\}$ is not realized by a (simply-connected) rational homology $\mathbb{CP}^2$.

\subsubsection{The Kervaire-Milnor theorem and the orthogonal complement argument}
In this subsection, we explore more refined conditions applicable to certain types that are not obstructed by the criterion derived from Donaldson's diagonalization theorem, i.e., the cases where there is an embedding from $\bigoplus_{i=1}^4 Q_{X(p_i,q_i)}$ into $-\mathbb{Z}^{n+1}$.

\begin{definition}[Essentially unique embedding]\label{def:unique_embedding} Let $X$ be a negative definite plumbed 4-manifold corresponding to the weighted graph shown in Figure \ref{fig:essentially_unique_embedding} which represents the boundary sum of four linear plumbed 4-manifolds. 

\begin{figure}[!th]
\centering
\begin{tikzpicture}[scale=0.9]
\draw (-2.5,0) node[circle, fill, inner sep=1.2pt, black]{};
\draw (-1,0) node[circle, fill, inner sep=1.2pt, black]{};
\draw (1,0) node[circle, fill, inner sep=1.2pt, black]{};
\draw (2.5,0) node[circle, fill, inner sep=1.2pt, black]{};
\draw (-2.5,-1.5) node[circle, fill, inner sep=1.2pt, black]{};
\draw (-1,-1.5) node[circle, fill, inner sep=1.2pt, black]{};
\draw (1,-1.5) node[circle, fill, inner sep=1.2pt, black]{};
\draw (2.5,-1.5) node[circle, fill, inner sep=1.2pt, black]{};
\draw (-2.5,-3) node[circle, fill, inner sep=1.2pt, black]{};
\draw (-1,-3) node[circle, fill, inner sep=1.2pt, black]{};
\draw (1,-3) node[circle, fill, inner sep=1.2pt, black]{};
\draw (2.5,-3) node[circle, fill, inner sep=1.2pt, black]{};
\draw (-2.5,-4.5) node[circle, fill, inner sep=1.2pt, black]{};
\draw (-1,-4.5) node[circle, fill, inner sep=1.2pt, black]{};
\draw (1,-4.5) node[circle, fill, inner sep=1.2pt, black]{};
\draw (2.5,-4.5) node[circle, fill, inner sep=1.2pt, black]{};

\draw (-2.5,0) node[above]{$-a_{1,1}$};
\draw (-1,0) node[above]{$-a_{1,2}$};
\draw (1,0) node[above]{$-a_{1,n_1-1}$};
\draw (2.5,0) node[above]{$-a_{1,n_1}$};
\draw (-2.5,-1.5) node[above]{$-a_{2,1}$};
\draw (-1,-1.5) node[above]{$-a_{2,2}$};
\draw (1,-1.5) node[above]{$-a_{2,n_2-1}$};
\draw (2.5,-1.5) node[above]{$-a_{2,n_2}$};
\draw (-2.5,-3) node[above]{$-a_{3,1}$};
\draw (-1,-3) node[above]{$-a_{3,2}$};
\draw (1,-3) node[above]{$-a_{3,n_3-1}$};
\draw (2.5,-3) node[above]{$-a_{3,n_3}$};
\draw (-2.5,-4.5) node[above]{$-a_{4,1}$};
\draw (-1,-4.5) node[above]{$-a_{4,2}$};
\draw (1,-4.5) node[above]{$-a_{4,n_4-1}$};
\draw (2.5,-4.5) node[above]{$-a_{4,n_4}$};

\draw (-2.5,0) node[below]{$v_{1,1}$};
\draw (-1,0) node[below]{$v_{1,2}$};
\draw (1,0) node[below]{$v_{1,n_1-1}$};
\draw (2.5,0) node[below]{$v_{1,n_1}$};
\draw (-2.5,-1.5) node[below]{$v_{2,1}$};
\draw (-1,-1.5) node[below]{$v_{2,2}$};
\draw (1,-1.5) node[below]{$v_{2,n_2-1}$};
\draw (2.5,-1.5) node[below]{$v_{2,n_2}$};
\draw (-2.5,-3) node[below]{$v_{3,1}$};
\draw (-1,-3) node[below]{$v_{3,2}$};
\draw (1,-3) node[below]{$v_{3,n_3-1}$};
\draw (2.5,-3) node[below]{$v_{3,n_3}$};
\draw (-2.5,-4.5) node[below]{$v_{4,1}$};
\draw (-1,-4.5) node[below]{$v_{4,2}$};
\draw (1,-4.5) node[below]{$v_{4,n_4-1}$};
\draw (2.5,-4.5) node[below]{$v_{4,n_4}$};

\draw (0,0) node{$\cdots$};
\draw (0,-1.5) node{$\cdots$};
\draw (0,-3) node{$\cdots$};
\draw (0,-4.5) node{$\cdots$};

\draw (-2.5,0)--(-0.5,0) (0.5,0)--(2.5,0)  (-2.5,-1.5)--(-0.5,-1.5) (0.5,-1.5)--(2.5,-1.5) (-2.5,-3)--(-0.5,-3) (0.5,-3)--(2.5,-3) (-2.5,-4.5)--(-0.5,-4.5) (0.5,-4.5)--(2.5,-4.5)  ;
\end{tikzpicture}
\caption{The graph of the boundary sum of four linear plumbed 4-manifolds.}
\label{fig:essentially_unique_embedding}
\end{figure}
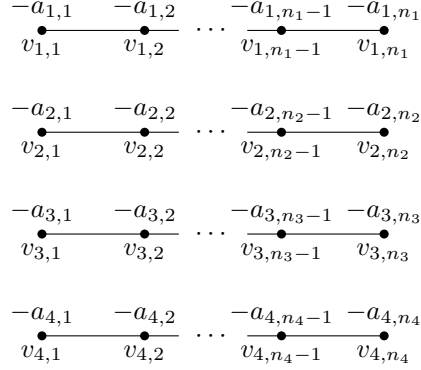

Suppose that we are given a lattice embedding $i:Q_X\hookrightarrow -\mathbb{Z}^n$ for some $n$. If $A:-\Bbb Z^n\to -\Bbb Z^n$ is a change of basis isomorphism (given by a matrix in $\textrm{O}(n)\cap \textrm{GL}(n,\Bbb Z)$), then $A\circ i$ is also an embedding $Q_X\hookrightarrow -\Bbb Z^n$. Moreover, for any subset $S$ of $\{1,2,3,4\}$, the map $i_S:Q_X\to -\Bbb Z^n$ defined by 
\[ 
    i_S(v_{j,k})=\begin{cases}
    -i(v_{j,k}) & \textrm{if $j\in S$}, \\ i(v_{j,k}) & \textrm{if $j\notin S$}, \end{cases} 
\]
is also an embedding. Define an equivalence relation $\sim$ on the set of all lattice embeddings $Q_X\hookrightarrow -\mathbb{Z}^n$ by declaring $i\sim A \circ i$ for any change of basis isomorphism $A$ and $i\sim i_S$ for any subset $S$ of $\{1,2,3,4\}$. We say that there is an \textit{essentially unique} embedding $Q_X\hookrightarrow -\mathbb{Z}^n$ if there is exactly one equivalence class of the set of all embeddings $Q_X\hookrightarrow -\mathbb{Z}^n$.
\end{definition}

For a compact oriented 4-manifold $X$, we define the \textit{signature}, $\sigma(X)$, of $X$ by $b_2^+(X)-b_2^-(X)$. The following is a classical result of Kervaire and Milnor, which is a generalization of Rokhlin's theorem.

\begin{theorem}[Kervaire-Milnor condition, \cite{KM-1961}]\label{thm:2.7} Let $X$ be a closed, oriented, smooth 4-manifold whose characteristic class can be represented by a smoothly embedded 2-sphere $\Sigma$ in $X$. Then
\[\sigma(X)\equiv[\Sigma]^2~\textup{mod}~16.\]
\end{theorem}

\begin{example}\label{ex:Kervaire-Milnor} Suppose that $S$ is a rational homology $\mathbb{CP}^2$ with four cyclic singularities of type $\{(2,1), (3,2), (11,2), (13,1)\}$. Let $M$ be a compact oriented smooth 4-manifold obtained from $S$ by removing the cone neighborhoods of the singular points. Let $X:=X(2,1) \amalg X(3,2) \amalg X(11,2) \amalg X(13,1)$. Then $Z=X\cup_{\partial}(-M)$ is a closed, oriented, negative definite 4-manifold with $b_2(Z)=7$. By Donaldson's theorem, $Q_Z$ is isomorphic to $-\Bbb Z^7$ and there is an embedding $i:Q_X\hookrightarrow Q_Z=-\Bbb Z^7$. In fact, it is easy to check that there is an essentially unique embedding $i$, which is given in Figure \ref{fig:example_Kerviare-Milnor} (indicated $i(v)$ at each vertex for the generator $v\in H_2(X;\mathbb{Z})$ corresponding the the vertex). Let $\Sigma_1$, $\Sigma_2$ be the smoothly embedded base 2-spheres in $X$ (and hence in $Z$) corresponding to the red vertices, respectively. These spheres are disjoint, so tubing them yields a smoothly embedded 2-sphere $\Sigma \subset Z$ with the homology class: 
    \[[\Sigma]=[\Sigma_1]+[\Sigma_2]=i(v_1)+i(v_6)=3e_1+e_2-e_3-e_4-e_5+e_6+e_7.\]
In particular, $[\Sigma]$ is characteristic as the coefficients of the $e_i$'s are all odd. However, we observe that 
\[[\Sigma]^2=-15 \not\equiv -7=\sigma(Z) \mod 16,\]
contradicting Theorem \ref{thm:2.7}. Therefore, it is proven that such an $S$ cannot exist.

\begin{figure}[!th]
\centering
\begin{tikzpicture}[scale=0.8]
\draw (0,0) node[circle, fill, inner sep=1.2pt, red]{};
\draw (-1,-1.5) node[circle, fill, inner sep=1.2pt, black]{};
\draw (1,-1.5) node[circle, fill, inner sep=1.2pt, black]{};
\draw (-1.75,-3) node[circle, fill, inner sep=1.2pt, black]{};
\draw (1.75,-3) node[circle, fill, inner sep=1.2pt, black]{};
\draw (0,-4.5) node[circle, fill, inner sep=1.2pt, red]{};
\draw (0,0) node[above]{$-2$};
\draw (-1,-1.5) node[above]{$-2$};
\draw (1,-1.5) node[above]{$-2$};
\draw (-1.75,-3) node[above]{$-6$};
\draw (1.75,-3) node[above]{$-2$};
\draw (0,-4.5) node[above]{$-13$};

\draw (-1,-1.5)--(1,-1.5) (-1.75,-3)--(1.75,-3) ;

\draw (0,0) node[below]{$e_1-e_2$};
\draw (-1,-1.5) node[below]{$e_3-e_4$};
\draw (1,-1.5) node[below]{$e_4-e_5$};
\draw (-2,-3) node[below]{$e_1+e_2+e_3+e_4+e_5-e_6$};
\draw (1.75,-3) node[below]{$e_6-e_7$};
\draw (0,-4.5) node[below]{$2e_1+2e_2-e_3-e_4-e_5+e_6+e_7$};
\end{tikzpicture}
\caption{An essentially unique embedding of $Q_{X(2,1)}\oplus Q_{X(3,2)}\oplus Q_{X(11,2)}\oplus Q_{X(13,1)}$ into $-\mathbb{Z}^{7}$.}
\label{fig:example_Kerviare-Milnor}
\end{figure}
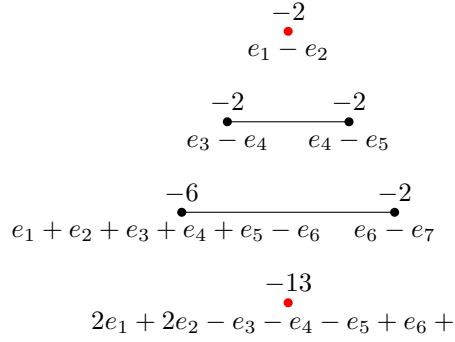
\end{example}

Next, we introduce an additional useful condition for the cases that are not obstructed by Donaldson's theorem. This condition leverages the condition that the first integral homology group of the smooth locus of our expected rational homology $\mathbb{CP}^2$ is trivial.

\begin{proposition}[{\cite[Lemma 2.4]{AMP-2022}}]\label{prop:2.8} Let $Y$ be an oriented rational homology 3-sphere which is the boundary of oriented compact 4-manifolds $X_1$ and $X_2$ with $H_1(X_1;\Bbb Z)=0$. If $X$ is a closed oriented 4-manifold obtained by $X:=X_1\cup_Y (-X_2)$, then the inclusions $X_1,-X_2\hookrightarrow X$ induce an embedding of lattices \[
  \iota:Q_{X_1}\oplus (-Q_{X_2}) \to Q_X
\]
such that $\iota(-Q_{X_2})$ is the orthogonal complement of $\iota(Q_{X_1})$ in $Q_X$.     
\end{proposition}

\begin{example}\label{ex:index} We can use the proposition above to demonstrate that a simply-connected rational homology $\mathbb{CP}^2$ with four cyclic singularities of type $\{(2,1),(3,2),(5,1),(9409,5519)\}$ does not exist. This particular type is not subject to obstruction by any other conditions discussed in this paper (see 
Example \ref{ex:9409}): 
Suppose $M$ is a simply-connected, smooth 4-manifold with $b_2(M)=b_2^+(M)=1$ and its boundary $\partial M$ is given by the disjoint union $L(2,1)\amalg L(3,2) \amalg L(5,1)\amalg L(9409, 5519)$. Define $X=X(2,1)\amalg X(3,2) \amalg X(5,1) \amalg X(9409,5519)$ and consider a closed negative definite 4-manifold $W:=X\cup_\partial (-M)$ with $b_2(W)=15$. Donaldson's diagonalization theorem implies that $Q_W$ is diagonalizable, and that there is an embedding $\iota:Q_X\oplus Q_{-M}\hookrightarrow Q_W\cong -\Bbb Z^{15}$. In particular, an essentially unique codimension one embedding $\iota|_{Q_X}:Q_X\hookrightarrow Q_W\cong -\Bbb Z^{15}$ can be observed, as indicated in Figure \ref{fig:embedding2}. 

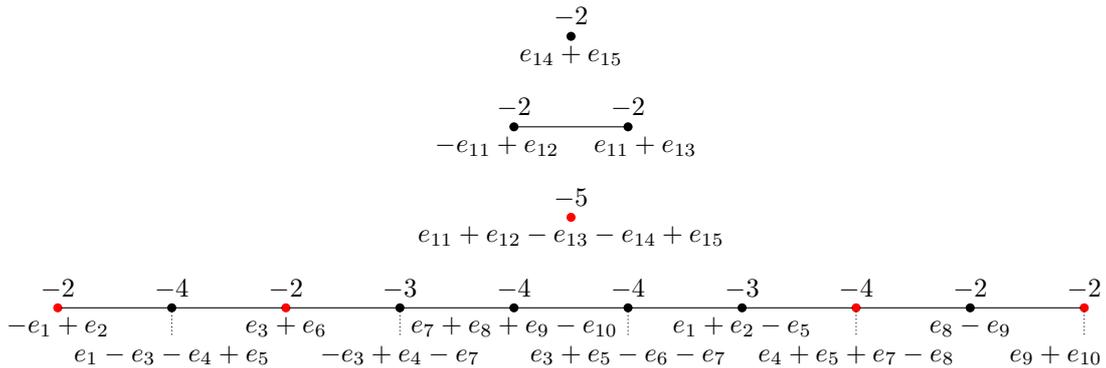
\begin{figure}[!th]
\centering
\begin{tikzpicture}[scale=0.75]
\draw (0,0) node[circle, fill, inner sep=1.2pt, black]{};
\draw (-1,-1.6) node[circle, fill, inner sep=1.2pt, black]{};
\draw (1,-1.6) node[circle, fill, inner sep=1.2pt, black]{};
\draw (0,-3.2) node[circle, fill, inner sep=1.2pt, red]{};
\draw (-9,-4.8) node[circle, fill, inner sep=1.2pt, red]{};
\draw (-7,-4.8) node[circle, fill, inner sep=1.2pt, black]{};
\draw (-5,-4.8) node[circle, fill, inner sep=1.2pt, red]{};
\draw (-3,-4.8) node[circle, fill, inner sep=1.2pt, black]{};
\draw (-1,-4.8) node[circle, fill, inner sep=1.2pt, black]{};
\draw (1,-4.8) node[circle, fill, inner sep=1.2pt, black]{};
\draw (3,-4.8) node[circle, fill, inner sep=1.2pt, black]{};
\draw (5,-4.8) node[circle, fill, inner sep=1.2pt, red]{};
\draw (7,-4.8) node[circle, fill, inner sep=1.2pt, black]{};
\draw (9,-4.8) node[circle, fill, inner sep=1.2pt, red]{};

\draw (0,0) node[above]{$-2$};
\draw (-1,-1.6) node[above]{$-2$};
\draw (1,-1.6) node[above]{$-2$};
\draw (0,-3.2) node[above]{$-5$};
\draw (-9,-4.8) node[above]{$-2$};
\draw (-7,-4.8) node[above]{$-4$};
\draw (-5,-4.8) node[above]{$-2$};
\draw (-3,-4.8) node[above]{$-3$};
\draw (-1,-4.8) node[above]{$-4$};
\draw (1,-4.8) node[above]{$-4$};
\draw (3,-4.8) node[above]{$-3$};
\draw (5,-4.8) node[above]{$-4$};
\draw (7,-4.8) node[above]{$-2$};
\draw (9,-4.8) node[above]{$-2$};

\draw (-1,-1.6)--(1,-1.6) (-8.94,-4.8)--(-5.06,-4.8) (-4.94,-4.8)--(4.94,-4.8) (5.06,-4.8)--(8.94,-4.8);

\draw[densely dotted] (-7,-4.8)--(-7,-5.3) (-3,-4.8)--(-3,-5.3) (1,-4.8)--(1,-5.3) (5,-4.86)--(5,-5.3)  (9,-4.86)--(9,-5.3);

\draw (0,0) node[below]{$e_{14}+e_{15}$};
\draw (-1.3,-1.6) node[below]{$-e_{11}+e_{12}$};
\draw (1.3,-1.6) node[below]{$e_{11}+e_{13}$};
\draw (0,-3.2) node[below]{$e_{11}+e_{12}-e_{13}-e_{14}+e_{15}$};
\draw (-9, -4.8) node[below]{$-e_1+e_2$};
\draw (-7, -5.3) node[below]{$e_1-e_3-e_4+e_5$};
\draw (-5, -4.8) node[below]{$e_3+e_6$};
\draw (-3, -5.3) node[below]{$-e_3+e_4-e_7$};
\draw (-1, -4.8) node[below]{$e_7+e_8+e_9-e_{10}$};
\draw (1, -5.3) node[below]{$e_3+e_5-e_6-e_7$};
\draw (3, -4.8) node[below]{$e_1+e_2-e_5$};
\draw (5, -5.3) node[below]{$e_4+e_5+e_7-e_8$};
\draw (7, -4.8) node[below]{$e_8-e_9$};
\draw (8.5, -5.3) node[below]{$e_9+e_{10}$};
\end{tikzpicture}
\caption{An essentially unique embedding of $Q_{X(2,1)}\oplus Q_{X(3,2)} \oplus  Q_{X(5,1)} \oplus  Q_{X(9409,5519)}$ into $-\mathbb{Z}^{15}$.}
\label{fig:embedding2}
\end{figure}

From the embedding, it is evident that the orthogonal complement $\iota(Q_X)^\perp$ of $\iota(Q_X)$ in $Q_W\cong -\Bbb Z^{15}$ is generated by the vector $2e_{11}+2e_{12}-2e_{13}+3e_{14}-3e_{15}$. In particular, $\det \left(\iota(Q_X)^\perp\right) = |2e_{11}+2e_{12}-2e_{13}+3e_{14}-3e_{15}|^2=-30$. However, we also have $\det \left(\iota(Q_{-M})\right)=\det Q_{-M}=\det(-Q_M)=-2\cdot 3\cdot 5\cdot 9409$. This implies that $\iota(Q_{-M})$ is properly contained within $\iota(Q_X)^\perp$. (The index $[\iota(Q_X)^\perp:\iota(Q_{-M})]$ is exactly 9409.) Consequently, such a smooth 4-manifold $M$ cannot exist by Proposition \ref{prop:2.8}.

We also note that tubing the 2-spheres corresponding to the red vertices in Figure \ref{fig:embedding2} results in a smoothly embedded characteristic 2-sphere in $Y$. This satisfies the condition in Theorem \ref{thm:2.7}.
\end{example}

\begin{remark}
In the example above, we observe that $9409 = 97^2$, which implies that the determinant $\det Q_X = 2 \cdot 3 \cdot 5 \cdot 9409$ is not square-free. The scenario described in the example does not occur if the boundary of $M$, denoted as $\partial M = L(p_1,q_1) \amalg L(p_2,q_2) \amalg L(p_3,q_3) \amalg L(p_4,q_4)$, satisfies the condition where $(p_i, p_j) = 1$ for $i \neq j$ and the product $p_1p_2p_3p_4$ is square-free.
\end{remark}

\subsubsection{Heegaard Floer $d$-invariants}
Another important condition arises from Heegaard Floer theory, introduced by Ozsv\'ath and Szab\'o \cite{Ozsvath-Szabo-2003}. A closed oriented 3-manifold $Y$ is defined as an \textit{L-space} if the rank of the hat version of Heegaard Floer homology of $Y$, denoted by $\textrm{rank}(\widehat{HF}(Y))$, equals $|H_1(Y;\Bbb Z)|$. Lens spaces are classified as $L$-spaces, and the property of being an $L$-space is preserved under connected sums \cite{Ozsvath-Szabo-2005}. 

A particularly relevant invariant in Heegaard Floer homology, which encapsulates information about 4-dimensional manifolds bounded by $Y$, is the \textit{d-invariant}, also known as the \textit{correction term}. This rational-valued invariant applies to a pair $(Y,\mathfrak{t})$, where $Y$ is a closed, oriented rational homology 3-sphere and $\mathfrak{t}$ is a spin$^c$ structure on $Y$. For further details, see \cite{Ozsvath-Szabo-2003}. In particular, for $d$-invariants of the boundary $3$-manifolds of a spin cobordism, especially those with small $b_2$, we have the following.


\begin{proposition}[{\cite[Lemma 2.7]{Lidman-Moore-Vazquez-2019}}]\label{prop:spin_cobordism} Let $(W,\mathfrak{s}):(Y,\mathfrak{t})\to (Y',\mathfrak{t}')$ be a spin cobordism between $L$-spaces satisfying $b_2^+(W)=1$ and $b_2^-(W)=0$. Then \[d(Y',\mathfrak{t}')-d(Y,\mathfrak{t})=-\frac{1}{4}.  \]    
\end{proposition}

The following is an immediate corollary from the fact $d(S^3,\mathfrak{t})=0$, where $\mathfrak{t}$ is the unique spin$^c$ structure on $S^3$.

\begin{corollary}[The spin $d$-invariant condition]\label{cor:2.10} Let $(W,\mathfrak{s})$ be a compact oriented spin smooth 4-manifold whose boundary $Y$ is an $L$-space. If $b_2^+(W)=1$ and $b_2^-(W)=0$, then $d(Y,\mathfrak{s}|_Y)=-\frac{1}{4}$.  
\end{corollary}

Suppose $Y:=L(p_1,q_1)\# \cdots \# L(p_4,q_4)$ bounds a compact, oriented, smooth 4-manifold $M$ with $H_1(M;\mathbb{Z})=0$ and $b_2(M)=b_2^+(M)=1$. The intersection form of $M$ is represented by the matrix $(p_1p_2p_3p_4)$ as stated in Lemma \ref{lem:linking_form} (1). If $p_1p_2p_3p_4$ is even, then $M$ possesses a spin structure \cite[Corollary 5.7.6]{Gompf-Stipsicz-1999}, and Corollary \ref{cor:2.10} implies that $d(Y,\mathfrak{s})=-\frac{1}{4}$ for some spin structure $\mathfrak{s}$ of $Y$. From this, we infer that if a singularity type $\left\{(p_1,q_1),\dots,(p_4,q_4)\right\}$ with even $p_1p_2p_3p_4$ is realized by a rational homology $\mathbb{CP}^2$ whose smooth locus has trivial first integral homology group, then $d(Y,\mathfrak{s})=-\frac{1}{4}$ for some spin structure $\mathfrak{s}$. In such instances, we state that the type $\left\{(p_1,q_1),\dots ,(p_4,q_4)\right\}$ \emph{satisfies the spin d-invariant condition}. Since the $d$-invariant is additive under the connected sum operation \cite[Theorem 4.3]{Ozsvath-Szabo-2003}, it follows that $d(Y,\mathfrak{s})=\sum_{i=1}^4 (d(L(p_i,q_i),\mathfrak{s}_i)$ if $\mathfrak{s}=\mathfrak{s}_1\#\cdots \# \mathfrak{s}_4$. Note that every spin structure $\mathfrak{s}$ of $Y$ is expressible in the form $\mathfrak{s}_1\#\cdots \# \mathfrak{s}_4$.

Recall that the orientation of $L(p,q)$ is defined so that it is obtained by $-(p/q)$-surgery along the unknot in $S^3$. Following \cite[Proposition 4.8]{Ozsvath-Szabo-2003} (their orientation convention is different from ours), we can choose an identification $\text{Spin}^c(L(p,q))\cong \Bbb Z_p$ such that the following recursive formula for the $d$-invariants is applicable: 
\begin{equation}\label{eq:d-invariant}
    d(L(p,q),i)=\frac{1}{4}-\frac{(2i+1-p-q)^2}{4pq}-d(L(q,r),j),
\end{equation}
assuming $0<q<p$ and $0\leq i<p+q$, where $r$ and $j$ are the reductions modulo $q$ of $p$ and $i$ respectively. Note that the spin structures of $L(p,q)$ correspond exactly to the integers among $(q-1)/2$ and $(p+q-1)/2$. In particular, if $p$ is odd, then $L(p,q)$ possesses a unique spin structure, which can also be inferred from the fact that $H^1(L(p,q);\Bbb Z_2)=0$.

\begin{example} We have $d(L(2,1),0)=-\frac{1}{4}$, $d(L(2,1),1)=\frac{1}{4}$, $d(L(3,2),2)=\frac{1}{2}$, $d(L(5,1),0)=-1$, $d(L(5,2),3)=0$, $d(L(5,4),4)=1$, and $d(L(p_4,1),0)=\displaystyle\frac{1-p_4}{4}$
from the recursive formula given above. Therefore, the $d$-invariant of $L(2,1)\# L(3,2) \# L(5,q_3) \# L(p_4,1)$ with $q_3=1,2$ or $4$ and $p_4>7$ cannot be $-\frac{1}{4}$ for all spin structures. Thus, the type $\{(2,1), (3,2), (5,q_3), (p_4,1)\}$ with $q_3=1,2,$ or $4$ and $p_4>7$ cannot be realized by a rational homology $\mathbb{CP}^2$ whose smooth locus has trivial first integral homology group. This, in fact, leads to the following corollary (cf. \cite[Lemma 2.17(1)]{Hwang-2021}). \end{example}

\begin{corollary} A singularity of type $(2,3,5,n)$ cannot be realized by a rational homology $\mathbb{CP}^2$ whose smooth locus has trivial first integral homology group if the fourth singularity is a rational double point.   
\end{corollary}
\begin{proof} For a singularity type $\{(2,1),(3,q_2),(5,q_3),(p_4,1)\}$, we must have $q_2=2$ by \cite[Lemma 5.3]{Hwang-Keum-2013}, and the argument above excludes all such types except the type $\{(2,1),(3,2),(5,4), (7,1)\}$. However, this type does not satisfy the linking form condition (Proposition \ref{prop:linking_form}). 
\end{proof}

\begin{example} Since $d(L(p_4,p_4-1),p_4-1)=\displaystyle\frac{p_4-1}{4}$, the same argument as in the example above also shows that the type $\{(2,1), (3,2), (5,q_3), (p_4,p_4-1)\}$ with $q_3=1,2,$ or $4$ and $p_4\geq 7$ cannot be realized by a rational homology $\mathbb{CP}^2$ whose smooth locus has trivial first integral homology group. In fact, Donaldson's theorem can be applied to demonstrate that such types cannot be realized by a rational homology $\mathbb{CP}^2$ (see Proposition \ref{lem:Donaldson_and_2's}).
\end{example}

\subsection{Independency of conditions}\label{subsec:independence}
In this subsection, we will briefly discuss the independency of conditions we have established, primarily obtained through examples. 

\begin{example}\label{ex:9409} Consider a singularity of type $\{(2,1),(3,2),(5,1),(9409, 5519)\}$ as in Example \ref{ex:index}. In this case, we have 
\[
    K_S^2=\frac{1210}{28227}< 3\cdot \frac{9439}{282270}=3e_{\textrm{orb}}(S),
\] 
so that the oBMY inequality is satisfied. Since 
    $-q=-551339\equiv 22889^2 \mod 282270$,
the linking form condition is also met. Furthermore, we have $D=12100=110^2$ and 
\[
    d(L(2,1),1)+d(L(3,2),2)+d(L(5,1),0)+d(L(9409, 5519),2759)=\frac{1}{4}+\frac{1}{2}+(-1)+0=-\frac{1}{4}, 
\] 
indicating that both the perfect squareness of $D$ (Proposition \ref{prop:D_square}) and the spin $d$-invariant condition (Corollary \ref{cor:2.10}) are satisfied. Therefore, this type can be obstructed solely by Proposition \ref{prop:2.8}. \end{example}

\begin{example}
A singularity of type $\left\{(2,1),(3,2),(5,1),(3529,1880)\right\}$ can be obstructed only by the squareness of $D$ condition (Proposition \ref{prop:D_square}). More precisely, we have 
\[
    K_S^2=\frac{874}{10587}< 3\cdot \frac{3559}{105870}=3e_{\textrm{orb}}(S)
\]
and 
\[
    -q=-201089\equiv 3521^2 \mod 105870.
\]
Furthermore, there exists an essentially unique lattice embedding of $Q_X$ into $-\mathbb{Z}^{17}$ as illustrated in Figure \ref{fig:embedding_D-square}, where \[X:=X(2,1)\natural X(3,2) \natural X(5,1)\natural X(3529,1880).\]
Tubing the 2-spheres corresponding to the red vertices results in a smoothly embedded characteristic 2-sphere, satisfying Theorem \ref{thm:2.7}. The orthogonal complement of $Q_X$ in $-\mathbb{Z}^{17}$ is generated by the vector 
\[
    15(e_1+e_2)-16(e_3+e_4-e_5)-24(e_6-e_7)+30(e_8+e_9-e_{10})+120(e_{11}+e_{12}+e_{13}+e_{14}+e_{15}+e_{16}-e_{17}),
\]
whose self-intersection number equals $-105870=-2\cdot 3\cdot 5\cdot 3529$, satisfying the condition of Proposition \ref{prop:2.8} (note that 3529 is a prime). Additionally, we have \[d(L(2,1),1)+d(L(3,2),2)+d(L(5,1),0)+d(L(3529,1880),2704)=\frac{1}{4}+\frac{1}{2}+(-1)+0=-\frac{1}{4}, \] 
thus fulfilling the spin $d$-invariant condition (Corollary \ref{cor:2.10}). However, \[
D=2\cdot 3\cdot 5\cdot 3529 \cdot \frac{874}{10587}=8740=2^2\cdot 5\cdot 19\cdot 23
\]
is not a square number.
\end{example}

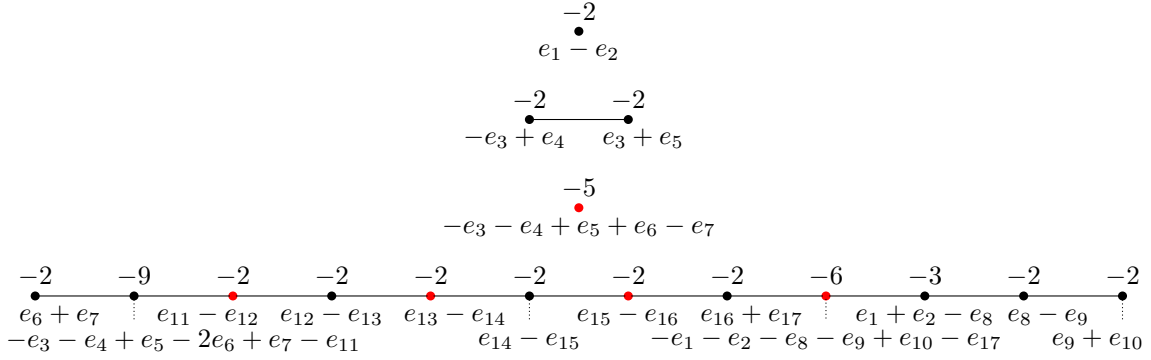
\begin{figure}[!th]
\centering
\begin{tikzpicture}[scale=0.65]
\draw (0,0) node[circle, fill, inner sep=1.2pt, black]{};
\draw (-1,-1.8) node[circle, fill, inner sep=1.2pt, black]{};
\draw (1,-1.8) node[circle, fill, inner sep=1.2pt, black]{};
\draw (0,-3.6) node[circle, fill, inner sep=1.2pt, red]{};
\draw (-9,-5.4) node[circle, fill, inner sep=1.2pt, black]{};
\draw (-7,-5.4) node[circle, fill, inner sep=1.2pt, red]{};
\draw (-5,-5.4) node[circle, fill, inner sep=1.2pt, black]{};
\draw (-3,-5.4) node[circle, fill, inner sep=1.2pt, red]{};
\draw (-1,-5.4) node[circle, fill, inner sep=1.2pt, black]{};
\draw (1,-5.4) node[circle, fill, inner sep=1.2pt, red]{};
\draw (3,-5.4) node[circle, fill, inner sep=1.2pt, black]{};
\draw (5,-5.4) node[circle, fill, inner sep=1.2pt, red]{};
\draw (7,-5.4) node[circle, fill, inner sep=1.2pt, black]{};
\draw (9,-5.4) node[circle, fill, inner sep=1.2pt, black]{};
\draw (11,-5.4) node[circle, fill, inner sep=1.2pt, black]{};
\draw (-11,-5.4) node[circle, fill, inner sep=1.2pt, black]{};

\draw (0,0) node[above]{$-2$};
\draw (-1,-1.8) node[above]{$-2$};
\draw (1,-1.8) node[above]{$-2$};
\draw (0,-3.6) node[above]{$-5$};
\draw (-9,-5.4) node[above]{$-9$};
\draw (-7,-5.4) node[above]{$-2$};
\draw (-5,-5.4) node[above]{$-2$};
\draw (-3,-5.4) node[above]{$-2$};
\draw (-1,-5.4) node[above]{$-2$};
\draw (1,-5.4) node[above]{$-2$};
\draw (3,-5.4) node[above]{$-2$};
\draw (5,-5.4) node[above]{$-6$};
\draw (7,-5.4) node[above]{$-3$};
\draw (9,-5.4) node[above]{$-2$};
\draw (11,-5.4) node[above]{$-2$};
\draw (-11,-5.4) node[above]{$-2$};

\draw (-1,-1.8)--(1,-1.8) (-11,-5.4)--(-7.06,-5.4) (-6.94,-5.4)--(-3.06,-5.4) (-2.94,-5.4)--(0.94,-5.4) (1.06,-5.4)--(4.94,-5.4) (5.06,-5.4)--(11,-5.4);

\draw[densely dotted] (-9,-5.4)--(-9,-5.9) (-1,-5.4)--(-1,-5.9) (5,-5.46)--(5,-5.9) (11,-5.4)--(11,-5.9);

\draw (0,0) node[below]{$e_1-e_2$};
\draw (-1.3,-1.8) node[below]{$-e_3+e_4$};
\draw (1.3,-1.8) node[below]{$e_3+e_5$};
\draw (0,-3.6) node[below]{$-e_3-e_4+e_5+e_6-e_7$};
\draw (-10.5, -5.4) node[below]{$e_6+e_7$};
\draw (-8, -5.9) node[below]{$-e_3-e_4+e_5-2e_6+e_7-e_{11}$};
\draw (-7.5, -5.4) node[below]{$e_{11}-e_{12}$};
\draw (-5, -5.4) node[below]{$e_{12}-e_{13}$};
\draw (-2.5, -5.4) node[below]{$e_{13}-e_{14}$};
\draw (-1, -5.9) node[below]{$e_{14}-e_{15}$};
\draw (1, -5.4) node[below]{$e_{15}-e_{16}$};
\draw (3.5, -5.4) node[below]{$e_{16}+e_{17}$};
\draw (5, -5.9) node[below]{$-e_1-e_2-e_8-e_9+e_{10}-e_{17}$};
\draw (7, -5.4) node[below]{$e_1+e_2-e_8$};
\draw (9.5, -5.4) node[below]{$e_8-e_9$};
\draw (10.5, -5.9) node[below]{$e_9+e_{10}$};
\end{tikzpicture}
\caption{An essentially unique embedding of $Q_{X(2,1)}\oplus Q_{X(3,2)} \oplus  Q_{X(5,1)} \oplus  Q_{X(3529,1880)}$ into $-\mathbb{Z}^{17}$.}
\label{fig:embedding_D-square}
\end{figure}

\begin{example}\label{example:4771} 
Consider a singularity of type $\{(2,1),(3,2),(5,1),(4771,634)\}$. We have 
\[
    K_S^2=\frac{640}{14313}< 3\cdot \frac{4801}{143130}=3e_{\textrm{orb}}(S),
\]
\[
    D=2\cdot 3\cdot 5\cdot 4771 \cdot \frac{640}{14313}=6400=80^2,
\]
and
\[
    -q=-214631\equiv 4763^2 \mod 143130.
\]
Also, there is an essentially unique lattice embedding of $Q_X$ into $-\mathbb{Z}^{18}$ as depicted in Figure \ref{fig:4771}, where \[X=X(2,1)\natural X(3,2) \natural X(5,1)\natural X(4771,634).\]
The orthogonal complement of $Q_X$ in $-\mathbb{Z}^{18}$ is generated by the vector 
\[
    8(e_1+e_2-e_3)+30(e_4-e_5)-45(e_6+e_7)+48e_8+72e_9+120(e_{10}+\cdots+e_{17}-e_{18}),
\]
whose self-intersection number equals $-143130=-2\cdot 3\cdot 5\cdot 4771$, satisfying the condition of Proposition \ref{prop:2.8} (note that $4771=13\cdot 367$). However, tubing the 2-spheres corresponding to the red vertices in Figure \ref{fig:4771} yields a smoothly embedded characteristic 2-sphere $\Sigma$ with 
\[ 
    [\Sigma]^2=-26 \not\equiv -18 \mod 16,
\] 
so that this type can be obstructed by Theorem \ref{thm:2.7}. Additionally, this type can be also obstructed by the spin $d$-invariant condition, as $d(L(4771,634),2702)=-\frac{3}{2}$.

On the other hand, the following example illustrates that the condition from the Kervaire-Milnor theorem and the spin $d$-invariant condition are independent of each other. 
\end{example}

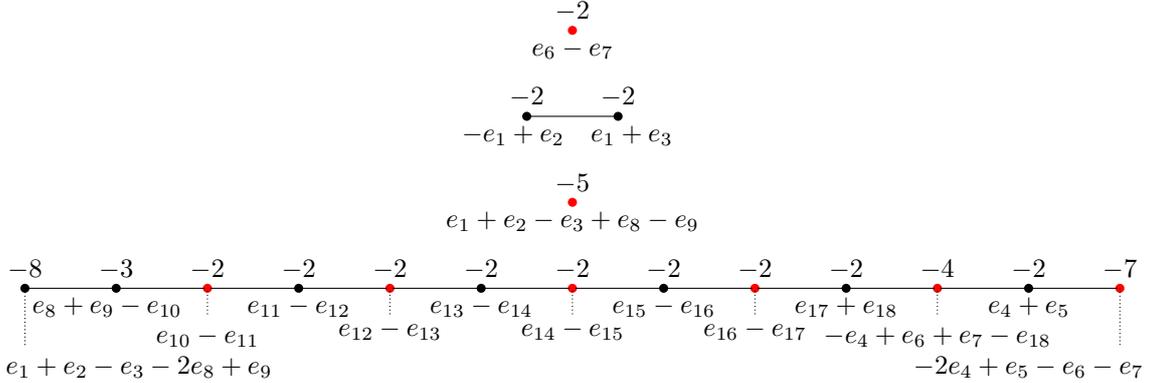
\begin{figure}[!th]
\centering
\begin{tikzpicture}[scale=0.6]
\draw (0,0) node[circle, fill, inner sep=1.2pt, red]{};
\draw (-1,-1.9) node[circle, fill, inner sep=1.2pt, black]{};
\draw (1,-1.9) node[circle, fill, inner sep=1.2pt, black]{};
\draw (0,-3.8) node[circle, fill, inner sep=1.2pt, red]{};
\draw (-12,-5.7) node[circle, fill, inner sep=1.2pt, black]{};
\draw (-10,-5.7) node[circle, fill, inner sep=1.2pt, black]{};
\draw (-8,-5.7) node[circle, fill, inner sep=1.2pt, red]{};
\draw (-6,-5.7) node[circle, fill, inner sep=1.2pt, black]{};
\draw (-4,-5.7) node[circle, fill, inner sep=1.2pt, red]{};
\draw (-2,-5.7) node[circle, fill, inner sep=1.2pt, black]{};
\draw (0,-5.7) node[circle, fill, inner sep=1.2pt, red]{};
\draw (2,-5.7) node[circle, fill, inner sep=1.2pt, black]{};
\draw (4,-5.7) node[circle, fill, inner sep=1.2pt, red]{};
\draw (6,-5.7) node[circle, fill, inner sep=1.2pt, black]{};
\draw (8,-5.7) node[circle, fill, inner sep=1.2pt, red]{};
\draw (10,-5.7) node[circle, fill, inner sep=1.2pt, black]{};
\draw (12,-5.7) node[circle, fill, inner sep=1.2pt, red]{};

\draw (0,0) node[above]{$-2$};
\draw (-1,-1.9) node[above]{$-2$};
\draw (1,-1.9) node[above]{$-2$};
\draw (0,-3.8) node[above]{$-5$};
\draw (-12,-5.7) node[above]{$-8$};
\draw (-10,-5.7) node[above]{$-3$};
\draw (-8,-5.7) node[above]{$-2$};
\draw (-6,-5.7) node[above]{$-2$};
\draw (-4,-5.7) node[above]{$-2$};
\draw (-2,-5.7) node[above]{$-2$};
\draw (0,-5.7) node[above]{$-2$};
\draw (2,-5.7) node[above]{$-2$};
\draw (4,-5.7) node[above]{$-2$};
\draw (6,-5.7) node[above]{$-2$};
\draw (8,-5.7) node[above]{$-4$};
\draw (10,-5.7) node[above]{$-2$};
\draw (12,-5.7) node[above]{$-7$};

\draw (-1,-1.9)--(1,-1.9) 
(-12,-5.7)--(-8.06,-5.7) (-7.94,-5.7)--(-4.06,-5.7) (-3.94,-5.7)--(-0.06,-5.7)
(0.06,-5.7)--(3.94,-5.7)
(7.94,-5.7)--(4.06,-5.7)
(11.94,-5.7)--(8.06,-5.7);

\draw[densely dotted] (-12,-5.7)--(-12,-7) 
(-8,-5.76)--(-8,-6.3)
(-4,-5.76)--(-4,-6.3)
(0,-5.76)--(0,-6.3)
(4,-5.76)--(4,-6.3)
(8,-5.76)--(8,-6.3)
(12,-5.7)--(12,-7);

\draw (0,0) node[below]{$e_6-e_7$};
\draw (-1.3,-1.9) node[below]{$-e_1+e_2$};
\draw (1.3,-1.9) node[below]{$e_{1}+e_{3}$};
\draw (0,-3.8) node[below]{$e_{1}+e_{2}-e_{3}+e_{8}-e_{9}$};
\draw (-9.5,-7) node[below]{$e_1+e_2-e_3-2e_8+e_9$};
\draw (-10.2,-5.7) node[below]{$e_8+e_9-e_{10}$};
\draw (-8,-6.35) node[below]{$e_{10}-e_{11}$};
\draw (-6,-5.7) node[below]{$e_{11}-e_{12}$};
\draw (-4,-6.2) node[below]{$e_{12}-e_{13}$};
\draw (-2,-5.7) node[below]{$e_{13}-e_{14}$};
\draw (0,-6.2) node[below]{$e_{14}-e_{15}$};
\draw (2,-5.7) node[below]{$e_{15}-e_{16}$};
\draw (4,-6.2) node[below]{$e_{16}-e_{17}$};
\draw (6,-5.7) node[below]{$e_{17}+e_{18}$};
\draw (8,-6.35) node[below]{$-e_4+e_6+e_7-e_{18}$};
\draw (10,-5.7) node[below]{$e_4+e_5$};
\draw (10,-7) node[below]{$-2e_4+e_5-e_6-e_7$};
\end{tikzpicture}
\caption{An essentially unique embedding of $Q_{X(2,1)}\oplus Q_{X(3,2)} \oplus  Q_{X(5,1)} \oplus  Q_{X(4771,634)}$ into $-\mathbb{Z}^{18}$.}
\label{fig:4771}
\end{figure}

\begin{example} (1) It is straightforward to demonstrate that a singularity of type 
\[ \{(2,1),(3,1),(7,3),(13,1)\} \]
is obstructed by the spin $d$-invariant condition, while it is not obstructed by the Kervaire-Milnor theorem (Theorem \ref{thm:2.7}). Figure \ref{fig:(13,1)} indicates an essentially unique embedding $Q_X\hookrightarrow -\mathbb{Z}^7$ and a characteristic $2$-sphere, where $X=X(2,1)\natural X(3,1)\natural X(7,3)\natural X(13,1).$
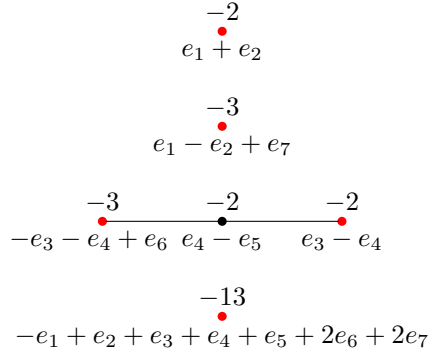
\begin{figure}[!th]
\centering
\begin{tikzpicture}[scale=0.7]
\draw (0,0) node[circle, fill, inner sep=1.2pt, red]{};
\draw (0,-1.8) node[circle, fill, inner sep=1.2pt, red]{};
\draw (-2.25,-3.6) node[circle, fill, inner sep=1.2pt, red]{};
\draw (0,-3.6) node[circle, fill, inner sep=1.2pt, black]{};
\draw (2.25,-3.6) node[circle, fill, inner sep=1.2pt, red]{};
\draw (0,-5.4) node[circle, fill, inner sep=1.2pt, red]{};

\draw (0,0) node[above]{$-2$};
\draw (0,-1.8) node[above]{$-3$};
\draw (-2.25,-3.6) node[above]{$-3$};
\draw (0,-3.6) node[above]{$-2$};
\draw (2.25,-3.6) node[above]{$-2$};
\draw (0,-5.4) node[above]{$-13$};

\draw (-2.19,-3.6)--(2.19,-3.6) ;

\draw (0,0) node[below]{$e_1+e_2$};
\draw (0,-1.8) node[below]{$e_1-e_2+e_7$};
\draw (-2.5,-3.6) node[below]{$-e_3-e_4+e_6$};
\draw (0,-3.6) node[below]{$e_4-e_5$};
\draw (2.25,-3.6) node[below]{$e_3-e_4$};
\draw (0,-5.4) node[below]{$-e_1+e_2+e_3+e_4+e_5+2e_6+2e_7$};
\end{tikzpicture}
\caption{An essentially unique embedding of $Q_{X(2,1)}\oplus Q_{X(3,1)} \oplus  Q_{X(7,3)} \oplus  Q_{X(13,1))}$ into $-\mathbb{Z}^{7}$.}
\label{fig:(13,1)}
\end{figure}

(2) Next, let us consider a singularity of type $\{(2,1),(3,1),(7,2),(25,3)\}$. We calculate that \[d(L(2,1),1)+d(L(3,1),0)+d(L(7,2),4)+d(L(25,3),2)=\frac{1}{4}-\frac{1}{2}+\frac{1}{2}-\frac{1}{2}=-\frac{1}{4}, \] 
which indicates that this type satisfies the spin $d$-invariant condition. On the other hand, there exist exactly two essentially distinct embeddings $Q_X\hookrightarrow -\mathbb{Z}^8$, as depicted in Figure \ref{fig:(25,3)}, where $X=X(2,1)\natural X(3,1)\natural X(7,2)\natural X(25,3).$ In both scenarios, tubing the 2-spheres corresponding to the red vertices results in a smoothly embedded characteristic 2-sphere $\Sigma$ with \[
[\Sigma]^2=-16\not\equiv -8 \mod 16.
\]
Thus, this type can be obstructed by the Kervaire-Milnor theorem.
\begin{figure}[!th]
\centering
\begin{tikzpicture}[scale=0.7]
\draw (0,0) node[circle, fill, inner sep=1.2pt, red]{};
\draw (0,-1.8) node[circle, fill, inner sep=1.2pt, red]{};
\draw (-2,-3.6) node[circle, fill, inner sep=1.2pt, black]{};
\draw (2,-3.6) node[circle, fill, inner sep=1.2pt, black]{};
\draw (-3,-5.4) node[circle, fill, inner sep=1.2pt, red]{};
\draw (0,-5.4) node[circle, fill, inner sep=1.2pt, black]{};
\draw (3,-5.4) node[circle, fill, inner sep=1.2pt, red]{};

\draw (0,0) node[above]{$-2$};
\draw (0,-1.8) node[above]{$-3$};
\draw (-2,-3.6) node[above]{$-4$};
\draw (2,-3.6) node[above]{$-2$};
\draw (-3,-5.4) node[above]{$-9$};
\draw (0,-5.4) node[above]{$-2$};
\draw (3,-5.4) node[above]{$-2$};

\draw  (-2,-3.6)--(2,-3.6) (-2.94,-5.4)--(2.94,-5.4);

\draw (0,0) node[below]{$e_6-e_7$};
\draw (0,-1.8) node[below]{$e_4+e_5-e_8$};
\draw (-2,-3.6) node[below]{$e_5+e_6+e_7+e_8$};
\draw (2,-3.6) node[below]{$e_4-e_5$};
\draw (-3.5,-5.4) node[below]{$-2e_1-2e_2-e_3$};
\draw (0,-5.4) node[below]{$e_2-e_3$};
\draw (3,-5.4) node[below]{$e_1-e_2$};

\draw (10,0) node[circle, fill, inner sep=1.2pt, red]{};
\draw (10,-1.8) node[circle, fill, inner sep=1.2pt, red]{};
\draw (8,-3.6) node[circle, fill, inner sep=1.2pt, black]{};
\draw (12,-3.6) node[circle, fill, inner sep=1.2pt, black]{};
\draw (7,-5.4) node[circle, fill, inner sep=1.2pt, red]{};
\draw (10,-5.4) node[circle, fill, inner sep=1.2pt, black]{};
\draw (13,-5.4) node[circle, fill, inner sep=1.2pt, red]{};

\draw (10,0) node[above]{$-2$};
\draw (10,-1.8) node[above]{$-3$};
\draw (8,-3.6) node[above]{$-4$};
\draw (12,-3.6) node[above]{$-2$};
\draw (7,-5.4) node[above]{$-9$};
\draw (10,-5.4) node[above]{$-2$};
\draw (13,-5.4) node[above]{$-2$};

\draw (7.06,-5.4)--(12.94,-5.4) (8,-3.6)--(12,-3.6) ;

\draw[densely dotted] (7,-5.46)--(7,-5.9);

\draw (10,0) node[below]{$e_6-e_7$};
\draw (10,-1.8) node[below]{$e_6+e_7+e_8$};
\draw (8,-3.6) node[below]{$e_1+e_2+e_3+e_5$};
\draw (12,-3.6) node[below]{$e_4-e_5$};
\draw (7,-5.9) node[below]{$e_3-e_4-e_5+e_6+e_7-2e_8$};
\draw (10,-5.4) node[below]{$e_2-e_3$};
\draw (13,-5.4) node[below]{$e_1-e_2$};

\end{tikzpicture}
\caption{Two essentially distinct embeddings of $Q_{X(2,1)}\oplus Q_{X(3,1)} \oplus  Q_{X(7,2)} \oplus  Q_{X(25,3)}$ into $-\mathbb{Z}^{8}$.}
\label{fig:(25,3)}
\end{figure}
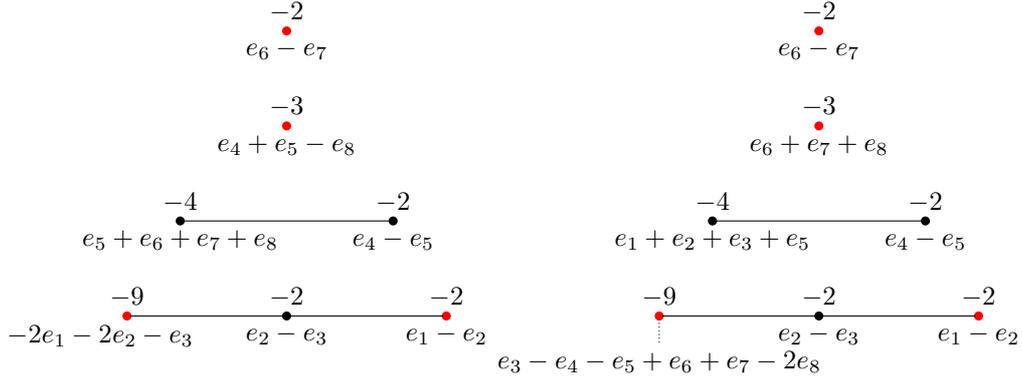  
\end{example}

\begin{example} The two principal conditions from the study of smooth 4-manifolds, one derived from Donaldson's diagonalization theorem and the other from Heegaard Floer theory, are shown to be independent of each other. As demonstrated in Example \ref{example:4771}, the singularity of type $\{(2,1),(3,2),(5,1),(4771,634)\}$ can be obstructed by spin $d$-invariant condition (Corollary \ref{cor:2.10}) but not by the Donaldson's theorem (Corollary \ref{cor:2.6}). Conversely, the singularity of type $\{(2,1),(3,2),(5,4), (43,26)\}$ can be obstructed by Corollary \ref{cor:2.6}, yet it does not meet the obstruction criteria of Corollary \ref{cor:2.10}.    
\end{example}

\section{Main Results}
Let $S$ be a rational homology $\mathbb{CP}^2$ with four cyclic singularities. Suppose that the smooth locus of $S$ is simply-connected and that $K_S$ is ample. As discussed in 
Section \ref{subsec:BMY}, the 4-tuple $(p_1,p_2,p_3,p_4)$ of the orders of the local fundamental groups of the singularities falls into one of the following three cases: 

\begin{enumerate}[label=Case (\arabic*), leftmargin=2cm]
    \item $(2,3,5,n)$, $n\geq 7 \ \mathrm{with} \ (n,30)=1$, 
    \item $(2,3,7,n)$, $n\in \{11,13,17,19,23,25,27,29,31,37,41\}$, 
    \item $(2,3,11,13)$.
\end{enumerate}

\subsection{Results for Case (2) and Case (3) }\label{subsec:cases_(2)_and_(3)}
In this subsection, we aim to demonstrate that the singularity types in Case (2) cannot be realized by a rational homology $\mathbb{CP}^2$. Furthermore, we will show that types in Case (3) are not realizable even in the smooth category. These provide the proof of Theorem \ref{thm:main-1} stated in the Introduction.

Most types in Case (2) and Case (3) are excluded from the algebraic Montgomery-Yang problem, with the exception of one type, as identified by Hwang and Keum. Recall that $L(p,q)$ is orientation-preserving homeomorphic to $L(p',q')$ if and only if $p=p'$ and $q'=q^{\pm 1}\mod p$. We identify $L(p,q)$ with $L(p,q^{-1})$, where $q^{-1}$ is the unique integer with $0<q^{-1}<p$ and $qq^{-1}=1\mod p$. 

For $p=2$, $L(2,1)$ is the only lens space. For $p=3$, there are two: $L(3,1), L(3,2)$; and for $p=7$ there are four: $L(7,1), L(7,2), L(7,3), \text{ and } L(7,6)$. For $p\in \{11,13,17,19,23,25,27,29,31,37,41\}$, the equation $x^2\equiv 1\mod p$ has exactly two different solutions modulo $p$, so there are $\frac{\phi(p)}{2}+1$ distinct lens spaces $L(p,q)$ where $\phi$ is Euler's totient function. It follows that there are exactly $\displaystyle 1\cdot 2\cdot 4\cdot \left(\frac{\phi(n)}{2}+1\right)$ distinct singularity types if $(p_1,p_2,p_3,p_4)=(2,3,7,n)$. Thus there are exactly $1008$ types in Case (2), and a similar calculation shows that there are exactly $84$ types in Case (3).

Among these $1008+84=1092$ types, exactly 24 types satisfy the squareness of $D$, the condition of Proposition \ref{prop:D_square}. These types are given in Table \ref{tab:1} (refer also to \cite[Section 5]{Hwang-Keum-2013}). Furthermore, among these $24$ types, only No.7 fulfills the (strong) orbifold BMY inequality (\ref{eq:oBMY}). This particular type is represented by $\{(2,1),(3,1),(7,6),(19,2)\}$. Consequently, to prove Theorem \ref{thm:main-1} in question, our focus narrows down to this specific type. We demonstrate that this singular type does not conform to the criteria established in Donaldson's diagonalization theorem.

\begin{table}[!th]
    \centering
    \begin{tabular}{c|c||c|c}
      No.  &  $(p_1,q_1), (p_2,q_2), (p_3,q_3), (p_4,q_4)$ & No. & $(p_1,q_1), (p_2,q_2), (p_3,q_3), (p_4,q_4)$ \\
      \hline 
       1 & (2,1), (3,2), (7,2), (11,2) & 13 & (2,1), (3,1), (7,2), (25,12) \\
       2 & (2,1), (3,2), (7,2), (11,7) & 14 & (2,1), (3,2), (7,1), (29,9) \\
       3 & (2,1), (3,1), (7,3), (13,3) & 15 & (2,1), (3,2), (7,2), (29,16) \\
       4 & (2,1), (3,1), (7,3), (13,4) & 16 & (2,1), (3,2), (7,2), (29,23) \\
       5 & (2,1), (3,1), (7,6), (13,12) & 17 & (2,1), (3,1), (7,3), (31,10) \\
       6 & (2,1), (3,1), (7,3), (19,3) & 18 & (2,1), (3,1), (7,6), (31,4) \\
       7 & (2,1), (3,1), (7,6), (19,2) & 19 & (2,1), (3,1), (7,6), (31,5) \\
       8 & (2,1), (3,1), (7,6), (19,8) & 20 & (2,1), (3,1), (7,6), (31,7) \\
       9 & (2,1), (3,1), (7,6), (19,14) & 21 & (2,1), (3,1), (7,2), (37,6) \\
       10 & (2,1), (3,2), (7,2), (23,13) & 22 & (2,1), (3,2), (7,3), (41,23) \\
       11 & (2,1), (3,1), (7,2), (25,6) & 23 & (2,1), (3,2), (11,2), (13,3) \\
       12 & (2,1), (3,1), (7,2), (25,11) & 24 & (2,1), (3,2), (11,2), (13,4) \\
    \end{tabular}
    \caption{All types in Cases (2) and (3) satisfying the squareness of $D$ condition.}
    \label{tab:1}
\end{table}

\begin{theorem}\label{thm:3.1}
    The singularity of type $\{(2,1),(3,1),(7,6),(19,2)\}$ cannot be realized by a rational homology $\mathbb{CP}^2$.
\end{theorem}

\begin{proof}
    Let $Y=L(2,1)\#L(3,1)\#L(7,6)\#L(19,2)$ and $X$ be the canonical negative definite plumbed 4-manifold bounded by $Y$, whose plumbing graph is given as in Figure \ref{fig:thm3.1}. By Corollary \ref{cor:2.6}, it is sufficient to demonstrate that $Q_X$ cannot be embedded into the lattice $-\mathbb{Z}^{11}$. 
    
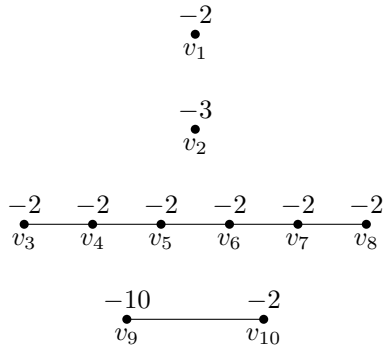
\begin{figure}[!th]
\centering
\begin{tikzpicture}[scale=0.9]
\draw (0,0) node[circle, fill, inner sep=1.2pt, black]{};
\draw (0,-1.4) node[circle, fill, inner sep=1.2pt, black]{};
\draw (-2.5,-2.8) node[circle, fill, inner sep=1.2pt, black]{};
\draw (-1.5,-2.8) node[circle, fill, inner sep=1.2pt, black]{};
\draw (-0.5,-2.8) node[circle, fill, inner sep=1.2pt, black]{};
\draw (0.5,-2.8) node[circle, fill, inner sep=1.2pt, black]{};
\draw (1.5,-2.8) node[circle, fill, inner sep=1.2pt, black]{};
\draw (2.5,-2.8) node[circle, fill, inner sep=1.2pt, black]{};
\draw (-1,-4.2) node[circle, fill, inner sep=1.2pt, black]{};
\draw (1,-4.2) node[circle, fill, inner sep=1.2pt, black]{};

\draw (0,0) node[above]{$-2$};
\draw (0,-1.4) node[above]{$-3$};
\draw (-2.5,-2.8) node[above]{$-2$};
\draw (-1.5,-2.8) node[above]{$-2$};
\draw (-0.5,-2.8) node[above]{$-2$};
\draw (0.5,-2.8) node[above]{$-2$};
\draw (1.5,-2.8) node[above]{$-2$};
\draw (2.5,-2.8) node[above]{$-2$};
\draw (-1,-4.2) node[above]{$-10$};
\draw (1,-4.2) node[above]{$-2$};

\draw (-2.5,-2.8)--(2.5,-2.8) (-1,-4.2)--(1,-4.2) ;

\draw (0,0) node[below]{$v_1$};
\draw (0,-1.4) node[below]{$v_2$};
\draw (-2.5,-2.8) node[below]{$v_3$};
\draw (-1.5,-2.8) node[below]{$v_4$};
\draw (-0.5,-2.8) node[below]{$v_5$};
\draw (0.5,-2.8) node[below]{$v_6$};
\draw (1.5,-2.8) node[below]{$v_7$};
\draw (2.5,-2.8) node[below]{$v_8$};
\draw (-1,-4.2) node[below]{$v_9$};
\draw (1,-4.2) node[below]{$v_{10}$};
\end{tikzpicture}
\caption{The plumbing graph of $X(2,1)\natural X(3,1)\natural X(7,6)\natural X(19,2).$}
\label{fig:thm3.1}
\end{figure}

Suppose there exists an embedding $\iota\colon Q_X\hookrightarrow -\Bbb Z^{11}$. For each vertex $v_i$ in the plumbing graph of $X$, let $v_i$ also denote the corresponding homology class in $H_2(X;\Bbb Z)$ associated with the sphere. 
Since every element in $-\Bbb Z^{11}$ of a square $-2$ can be expressed in the form $\pm e_i \pm e_j$ for some $1\leq i<j\leq 11$, we may assume, after an appropriate change of basis, that $\iota(v_1)=e_1-e_2$.

Next, consider $\iota(v_3)=\pm e_i \pm e_j$ for some $i<j$. If $i=1$, then $\iota(v_3)$ must be $\pm (e_1+e_2)$ because $\iota(v_1)\cdot \iota(v_3)=v_1\cdot v_3=0$. However, this leads to a contradiction when considering modulo 2, as we have $1=v_4\cdot v_3=\iota(v_4)\cdot \iota(v_3)=\iota(v_4) \cdot \pm (e_1+e_2) \equiv \iota(v_4)\cdot (e_1-e_2)=\iota(v_4)\cdot \iota(v_1)=v_4\cdot v_1=0$. Therefore, $i$ must be greater than $1$. If $i=2$, then $0=v_1\cdot v_3=\iota(v_1)\cdot \iota(v_3)=(e_1-e_2)\cdot (\pm e_2 \pm e_j)=\pm 1$, which is again a contradiction. Thus, we deduce that $i>2$, and, after another change of basis, we may assume $\iota(v_3)=e_3-e_4$. 

Now consider $\iota(v_4)$. Write $\iota(v_4)=\pm e_i \pm e_j$ with $i<j$. Following a similar argument as above, we must have $i>2$. Additionally, since $v_3\cdot v_4=1$, it is necessarily that $|\{i,j\}\cap \{3,4\}|=1$. After an appropriate change of basis, we may assume that $\iota(v_4)=e_4-e_5$. Proceeding with this process, we can further assume that $\iota(v_j)=e_j-e_{j+1}$ for $j=3,\dots,8$ and that $\iota(v_{10})=e_{10}-e_{11}$.

Finally, let us consider $\iota(v_2)$. Given its square is $-3$, it follows that $\iota(v_2)=\pm e_i\pm e_j \pm e_k$ for some $i<j<k$. However, it can be easily demonstrated that for any such choices of $i,j,k$, the condition $\iota(v_2)\cdot \iota(v_j)=0$ for all $j\in \{1,3,4,5,6,7,8,10\}$ cannot be satisfied. This leads to a contradiction, implying that such an embedding $\iota$ cannot exist. 
\end{proof}

\begin{remark} In addition to the previously discussed criteria, the type $\{(2,1),(3,1),(7,6),(19,2)\}$ can be also excluded by applying the spin $d$-invariant condition (Corollary \ref{cor:2.10}).
\end{remark}

\subsubsection*{Smooth category}
Within the 1092 types identified in Case (2) and Case (3), an intriguing question arises: Which of these types can be precluded from realizing a simply-connected rational homology $\mathbb{CP}^2$ using only topological and smooth conditions? Remarkably, all 84 types in Case (3) can be excluded by applying criteria from the theory of smooth 4-manifolds. This implies that none of the types in Case (3) are realizable as a simply-connected rational homology $\mathbb{CP}^2$ within the smooth category. Consequently, this resolves the case $(2,3,11,13)$ for the original Montgomery-Yang problem as stated in Conjecture \ref{conj:original_MY}.

\begin{theorem}\label{thm:(2,3,11,13)} The connected sum $L(p_1,q_1)\#\cdots \# L(p_4,q_4)$ of four lens spaces with $(p_1,p_2,p_3,p_4)=(2,3,11,13)$ cannot bound a compact oriented smooth 4-manifold with $b_2=b_2^+=1$ and trivial first integral homology group. 
\end{theorem}

\begin{proof} The 84 possible combinations of the $q_i$'s are given by 
\[
    (q_1,q_2,q_3,q_4)\in \{1\}\times \{1,2\}\times \{1,2,3,5,7,10\} \times \{1,2,3,4,5,6,12\},
\] 
and among which precisely $12$ cases satisfy the linking form condition (Proposition \ref{prop:linking_form}). These cases are \[
(q_1,q_2,q_3,q_4)\in \{1\}\times \{2\}\times \{2,7,10\} \times \{1,3,4,12\}.
\]
All of these 12 cases, with the exception of $(q_1,q_2,q_3,q_4)=(1,2,2,1)$, can be excluded based on Donaldson's diagonlization theorem (Corollary \ref{cor:2.6}).The reasoning behind these exclusions is similar to the argument presented in the proof of Theorem \ref{thm:3.1}. Therefore, we will specifically focus on the case $(q_1,q_2,q_3,q_4)=(1,2,2,3)$. In this scenario, the boundary sum $X:=X(2,1)\natural X(3,2)\natural X(11,2)\natural X(13,3)$ is represented by the plumbing diagram shown in Figure \ref{fig:thm3.2-(1)}.

\begin{figure}[!th]
\centering
\begin{tikzpicture}[scale=0.8]
\draw (0,0) node[circle, fill, inner sep=1.2pt, black]{};
\draw (-1,-1.4) node[circle, fill, inner sep=1.2pt, black]{};
\draw (1,-1.4) node[circle, fill, inner sep=1.2pt, black]{};
\draw (-1,-2.8) node[circle, fill, inner sep=1.2pt, black]{};
\draw (1,-2.8) node[circle, fill, inner sep=1.2pt, black]{};
\draw (-1.5,-4.2) node[circle, fill, inner sep=1.2pt, black]{};
\draw (0,-4.2) node[circle, fill, inner sep=1.2pt, black]{};
\draw (1.5,-4.2) node[circle, fill, inner sep=1.2pt, black]{};
\draw (0,0) node[above]{$-2$};
\draw (-1,-1.4) node[above]{$-2$};
\draw (1,-1.4) node[above]{$-2$};
\draw (-1,-2.8) node[above]{$-6$};
\draw (1,-2.8) node[above]{$-2$};
\draw (-1.5,-4.2) node[above]{$-5$};
\draw (0,-4.2) node[above]{$-2$};
\draw (1.5,-4.2) node[above]{$-2$};

\draw (-1,-1.4)--(1,-1.4) (-1,-2.8)--(1,-2.8)  (-1.5,-4.2)--(1.5,-4.2);

\draw (0,0) node[below]{$v_1$};
\draw (-1,-1.4) node[below]{$v_2$};
\draw (1,-1.4) node[below]{$v_3$};
\draw (-1,-2.8) node[below]{$v_4$};
\draw (1,-2.8) node[below]{$v_5$};
\draw (-1.5,-4.2) node[below]{$v_6$};
\draw (0,-4.2) node[below]{$v_7$};
\draw (1.5,-4.2) node[below]{$v_8$};
\end{tikzpicture}
\caption{The plumbing graph of $X(2,1)\natural X(3,2)\natural X(11,2)\natural X(13,3).$}
\label{fig:thm3.2-(1)}
\end{figure}
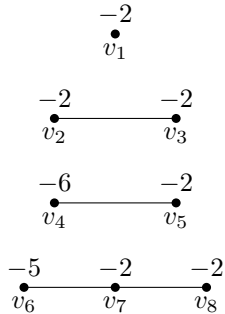

For each vertex $v_i$ in the plumbing graph of $X$, we also denote the corresponding homology class as $v_i$ in $H_2(X;\Bbb Z)$. Suppose that there exists an embedding $\iota:Q_X\hookrightarrow -\Bbb Z^9$. Since every element of square $-2$ in $-\Bbb Z^9$  can be expressed as $\pm e_i \pm e_j$ for some $1\leq i<j\leq 9$, after a change of basis, we may assume that $\iota(v_1)=e_1-e_2$. 

Next, let $\iota(v_2)=\pm e_i \pm e_j$ for some $i<j$. If $i=1$, then $\iota(v_2)$ must be $\pm(e_1+e_2)$ since $v_1\cdot v_2=0$. However, this leads to a contradiction, as $1=v_3\cdot v_2=\iota(v_3)\cdot \iota(v_2)= \iota(v_3)\cdot \iota(v_1)=v_3\cdot v_1=0$ modulo $2$. If $i=2$, then $0=v_1\cdot v_2=\iota(v_1)\cdot \iota(v_2)=\pm 1$, which is again a contradiction. Therefore, $i>2$, and after a change of basis, we may assume that $\iota(v_2)=e_3-e_4$. 

Now, let $\iota(v_3)=\pm e_i \pm e_j$ with $i<j$. As previously, $i>2$ must hold. Additionally, since $v_2\cdot v_3=1$, we must have $|\{i,j\} \cap \{3,4\}|=1$. Therefore, after a change of basis, we may assume that $\iota(v_3)=e_4-e_5$.

Proceeding similarly, let $\iota(v_5)=\pm e_i \pm e_j$ $(i<j)$. We should have $i>5$, so we may assume $\iota(v_5)=e_6-e_7$. Likewise, for  $\iota(v_7)=\pm e_i \pm e_j$ $(i<j)$ with $i>7$, we may assume $\iota(v_7)=e_8-e_9$. The only possibility for $\iota(v_8)$ then becomes $-(e_8+e_9)$, but this results in a contradiction: $1=v_6\cdot v_7=\iota(v_6)\cdot \iota(v_7)=\iota(v_6)\cdot \iota(v_8)=v_6\cdot v_8=0$ modulo $2$. We conclude that such an embedding $\iota:Q_X\hookrightarrow -\Bbb Z^9$ cannot exist.

Finally, the only remaining case to consider is: 
\[
    (q_1,q_2,q_3,q_4)=(1,2,2,1).
\]
We will demonstrate that this case adheres to Donaldson's theorem, yet it can be obstructed by applying the Kervaire-Milnor theorem (Theorem \ref{thm:2.7}). The boundary sum \[X':=X(2,1)\natural X(3,2)\natural X(11,2)\natural X(13,1)\] is represented by the plumbing description shown in Figure \ref{fig:thm3.2-(2)}. We claim that there exists an essentially unique embedding $\iota:Q_{X'}\hookrightarrow -\Bbb Z^7$, as defined in Definition \ref{def:unique_embedding}. Assuming the existence of such an embedding, and paralleling the approach in the case $(q_1,q_2,q_3,q_4)=(1,2,2,3)$, we may assume that $\iota(v_1)=e_1-e_2$, $\iota(v_2)=e_3-e_4$, $\iota(v_3)=e_4-e_5$, $\iota(v_5)=e_6-e_7$. 

Next, we consider $\iota(v_4)$. Given the orthogonality conditions $v_4\cdot v_1=v_4\cdot v_2=v_4\cdot v_3=0$, and representing $\iota(v_4)$ as $\sum_{t=1}^7 a_te_t$, it becomes apparent that $\iota(v_4)$ cannot be of the form $\pm 2e_i \pm e_j \pm e_k$. Instead, it must be of the form $\pm e_{i_1}\pm \cdots \pm e_{i_6}$. The constraints $a_1=a_2$ and $a_3=a_4=a_5$, coupled with the condition $v_4\cdot v_5=1$, uniquely determine $\iota(v_4)=e_1+e_2+e_3+e_4+e_5-e_6$ up to the equivalence relation of Definition \ref{def:unique_embedding}.

For $\iota(v_6)=\sum_{t=1}^7 b_te_t$, the orthogonality with $v_1,v_2,v_3,v_5$ leads to the conditions $b_1=b_2$, $b_3=b_4=b_5$, and $b_6=b_7$. Furthermore, the condition $v_6\cdot v_4=0$ implies $b_1+b_2+b_3+b_4+b_5=b_6$. These conditions uniquely determine $\iota(v_6)=2e_1+2e_2-e_3-e_4-e_5+e_6+e_7$, again up to the equivalence relation of Definition \ref{def:unique_embedding}.

This establishes the existence of an essentially unique lattice embedding $\iota:Q_{X'}\hookrightarrow -\Bbb Z^7$. The application of the Kervaire-Milnor theorem to this case is further elucidated in Example \ref{ex:Kervaire-Milnor}. 
\end{proof}

\begin{figure}[!th]
\centering
\begin{tikzpicture}[scale=0.8]
\draw (0,0) node[circle, fill, inner sep=1.2pt, black]{};
\draw (-1,-1.4) node[circle, fill, inner sep=1.2pt, black]{};
\draw (1,-1.4) node[circle, fill, inner sep=1.2pt, black]{};
\draw (-1,-2.8) node[circle, fill, inner sep=1.2pt, black]{};
\draw (1,-2.8) node[circle, fill, inner sep=1.2pt, black]{};
\draw (0,-4.2) node[circle, fill, inner sep=1.2pt, black]{};
\draw (0,0) node[above]{$-2$};
\draw (-1,-1.4) node[above]{$-2$};
\draw (1,-1.4) node[above]{$-2$};
\draw (-1,-2.8) node[above]{$-6$};
\draw (1,-2.8) node[above]{$-2$};
\draw (0,-4.2) node[above]{$-13$};

\draw (-1,-1.4)--(1,-1.4) (-1,-2.8)--(1,-2.8) ;

\draw (0,0) node[below]{$v_1$};
\draw (-1,-1.4) node[below]{$v_2$};
\draw (1,-1.4) node[below]{$v_3$};
\draw (-1,-2.8) node[below]{$v_4$};
\draw (1,-2.8) node[below]{$v_5$};
\draw (0,-4.2) node[below]{$v_6$};
\end{tikzpicture}
\caption{The plumbing graph of $X(2,1)\natural X(3,2)\natural X(11,2)\natural X(13,1).$}
\label{fig:thm3.2-(2)}
\end{figure}
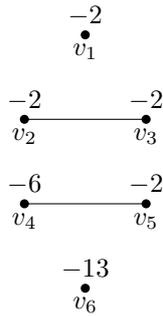
\begin{remark} The type $\{(2,1),(3,2),(11,2),(13,1)\}$ can be also excluded by the spin $d$-invariant condition (Corollary \ref{cor:2.10}).
\end{remark}

Next, we consider 1008 types in Case (2). Initially, after applying the linking form condition (Proposition \ref{prop:linking_form}), we narrow these down to 128 types. Among this subset, only 35 types meet the criteria set forth in Donaldson's theorem (Corollary \ref{cor:2.6}). Further refinement is achieved by employing Kervaire-Milnor (Theorem \ref{thm:2.7}) and Proposition \ref{prop:2.8}, which reduces the number to 21 types. Finally, by incorporating the spin $d$-invariant condition (Corollary \ref{cor:2.10}) into our analysis, we arrive at a conclusive set of 13 types, as detailed in Table \ref{tab:2}.
\begin{table}[t]
    \centering
    \begin{tabular}{c|c}
      No.  &  $(p_1,q_1), (p_2,q_2), (p_3,q_3), (p_4,q_4)$ \\
      \hline 
       1 & (2,1), (3,2), (7,1), (11,2)  \\
       2 & (2,1), (3,1), (7,3), (19,2)  \\
       3 & (2,1), (3,2), (7,1), (23,2)  \\
       4 & (2,1), (3,2), (7,1), (23,4)  \\
       5 & (2,1), (3,2), (7,2), (23,3) \\
       6 & (2,1), (3,1), (7,2), (25,2)  \\
       7 & (2,1), (3,2), (7,1), (29,4) \\
       8 & (2,1), (3,2), (7,2), (29,5) \\
       9 & (2,1), (3,1), (7,3), (31,2) \\
       10 & (2,1), (3,1), (7,3), (31,4)  \\
       11 & (2,1), (3,1), (7,2), (37,2) \\
       12 & (2,1), (3,1), (7,2), (37,8) \\
       13 & (2,1), (3,1), (7,2), (37,13) 
   \end{tabular} 
    \caption{All types in Case (2) satisfying topological and smooth conditions.}
    \label{tab:2}
\end{table}

\begin{theorem} The connected sum $L(p_1,q_1)\#\cdots \# L(p_4,q_4)$ of four lens spaces with $(p_1,p_2,p_3,p_4)=(2,3,7,n)$ and $n<43$, cannot bound a compact oriented smooth 4-manifold with $b_2=b_2^+=1$ and trivial first integral homology group, with the potential exception of the 13 cases enumerated in Table \ref{tab:2}.
\end{theorem}
Note that if $Y:=L(p_1,q_1)\# \cdots \# L(p_4,q_4)$ with $(p_1,p_2,p_3,p_4)=(2,3,7,n)$ bounds such a 4-manifold, then it is necessary that  $(n,42)=1$; otherwise $H_1(Y;\mathbb{Z})$ would not be cyclic, and this would contradict Lemma \ref{lem:linking_form} (1).

An interesting observation arises from the comparison of Table \ref{tab:1} and Table \ref{tab:2}: there are no common types between these two lists. Specifically, for the 24 types listed in Table \ref{tab:1}, Types No.11 and No.12 fail to satisfy the linking form condition (Proposition \ref{prop:linking_form}), while the remaining 22 types do not meet the criteria of Donaldson's diagonalization theorem (Corollary \ref{cor:2.6}). This outcome suggests that it is quite rare for a singularity type to simultaneously navigate through the obstructions posed by both algebraic geometry and smooth 4-manifold theory. In Case (1) where $(p_1,p_2,p_3,p_4)=(2,3,5,n)$, this rarity has been confirmed for small values of $n$ in the next section.

\subsection{Results for Case (1)}
Now, we turn our attention to the case $(p_1,p_2,p_3,p_4)=(2,3,5,n)$. According to \cite[Lemma 5.3]{Hwang-Keum-2013}, it is required that $q_2=2$ in this case.

\subsubsection{Experimental results}\label{subsec:experimental_result}
In Section \ref{subsec:cases_(2)_and_(3)}, we discussed Case (2) and Case (3), notably observing the lack of overlapping types between the lists in Table \ref{tab:1} and Table \ref{tab:2}. This observation suggests that the types of singularities that simultaneously satisfy conditions from both algebraic geometry and topological/smooth 4-manifolds theory are extremely rare, perhaps even non-existent. To further investigate this phenomenon, particularly in the context of Case (1), we developed a computer program to systematically verify whether a given type of singularity fulfills all the conditions outlined in Section \ref{sec:conditions}. A key aspect of our program involves using the \texttt{OrthogonalEmbeddings} function from \cite{GAP} to validate the Donaldson condition.

We conducted a thorough verification of our conditions for all singularities of type $(2,3,5,p_4)$ with $p_4<50000$. Although there were types that could not be excluded by any of the conditions discussed in this paper, such instances appear to be exceedingly rare. For example, out of more than $4.7\times 10^8$ types\footnote{Precisely, there are $3\times 158,353,370=475,060,110$ types.} of $\{(2,1), (3,2), (5,q_3),(p_4, q_4)\}$ with $p_4<50000$, $q_3=1,2$ or $4$, $(p_4, 30)=1$, and $(p_4, q_4)=1$ up to $q_4q_4'\equiv 1\mod{p_4}$, only $9767$ types\footnote{2579, 4146, and 3042 types for $q_3=1$, $2$, and $4$, respectively.} satisfy the orbifold BMY inequality, the perfect squareness of $D$, and the linking form condition. The 16 types listed in Table \ref{tab:3} are the only survivors after the application of further smooth conditions, such as Donaldson's theorem and the spin $d$-invariant condition. Since the first occurrence of such a type is observed at $p_4=2599$, this leads to Theorem \ref{thm:experimental}.

\begin{table}[t]
\centering
\begin{tabular}{c|c||c|c}
       No.  &  $(p_1,q_1), (p_2,q_2), (p_3,q_3), (p_4,q_4)$ & No. & $(p_1,q_1), (p_2,q_2), (p_3,q_3), (p_4,q_4)$ \\ \hline
       1 & (2,1), (3,2), (5,1), (2599,1384) & 9 & (2,1), (3,2), (5,1), (26869,14314) \\
       2 & (2,1), (3,2), (5,2), (2623,821) & 10 & (2,1), (3,2), (5,1), (27289,3616) \\
       3 & (2,1), (3,2), (5,2), (5203,1651) & 11 & (2,1), (3,2), (5,1), (31309,19161) \\
       4 & (2,1), (3,2), (5,1), (6049,3866) & 12 & (2,1), (3,2), (5,1), (32149,18482) \\
       5 & (2,1), (3,2), (5,2), (9607,946) & 13 & (2,1), (3,2), (5,2), (37837,3192) \\
       6 & (2,1), (3,2), (5,2), (12727,1884) & 14 & (2,1), (3,2), (5,1), (44161,27733) \\
       7 & (2,1), (3,2), (5,2), (17833,4898) & 15 & (2,1), (3,2), (5,2), (44407,11507) \\
       8 & (2,1), (3,2), (5,2), (26473,7271) & 16 & (2,1), (3,2), (5,4), (47929,9960)
    \end{tabular}         
    \caption{All singularities of type $(2,3,5,p_4)$ with $p_4<50000$ satisfying all conditions discussed in this paper.}
    \label{tab:3}    
\end{table}

This result prompts us to conjecture that there may only be a finite number of singularity types in the case $(p_1,p_2,p_3,p_4)=(2,3,5,n)$ that are not precluded by the techniques we have employed. However, in contrast to this conjecture, we have identified an infinite family of singularity types, as detailed in Proposition \ref{prop:mysterious}, that surprisingly cannot be excluded by any of the conditions discussed in this paper.

\subsubsection{The orbifold BMY inequality and Donaldson's theorem}\label{subsec:BMY_and_Donaldson}
In this subsection, we explore a methodology for identifying extensive, infinite families of singularity types with the 4-tuple \[(p_1,p_2,p_3,p_4)=(2,3,5,n)\] that are not precluded by conditions rooted in algebraic geometry, such as the orbifold BMY inequality, yet are readily obstructed by the criteria set forth in Donaldson's diagonalization theorem. This analysis sheds light on the distinct ways these singularity types interact with different mathematical frameworks.

We will now present an explicit methodology for identifying such infinite families of singularity types. We begin with a calculation of $K_S^2$ for the types in Case (1). Consider $S$, a simply-connected rational homology $\mathbb{CP}^2$ whose canonical divisor $K_S$ is ample, having four cyclic singularities of type $\{(2,1),(3,2),(5,q_3),(p_4,q_4)\}$. For $q_3$, we have three possible choices: $q_3=1,2,\text{ or }4$. Note that the cases $q_3=2$ and $q_3=3$ are equivalent, as $L(5,2)$ and $L(5,3)$ are orientation-preserving homeomorphic. Writing $\displaystyle p_4/(p_4-q_4)=[n_1,\dots,n_\ell]$, we can then apply the formula (\ref{eq:KS_square}) from Section \ref{subsec:computation_of_K_S^2} to compute $K_S^2$.
\[
    K_S^2=\begin{cases} \displaystyle\sum_{j=1}^\ell n_j -3\ell  +\frac{10}{3} - \frac{q_4+q_4^{-1}-2}{p_4}, & \textrm{if $q_3=1$}, \\
 \displaystyle\sum_{j=1}^\ell n_j -3\ell  +\frac{86}{15} - \frac{q_4+q_4^{-1}-2}{p_4}, & \textrm{if $q_3=2$}, \\
     \displaystyle\sum_{j=1}^\ell n_j -3\ell  +\frac{122}{15} - \frac{q_4+q_4^{-1}-2}{p_4}, & \textrm{if $q_3=4$}.
\end{cases} 
\]
On the other hand, we have $0<K_S^2<3e_{\textrm{orb}}(S)=\displaystyle\frac{1}{10}+\frac{3}{p_4}$ by Theorem \ref{thm:oBMY} (1), and it is evident that $0\leq \displaystyle\frac{q_4+q_4^{-1}-2}{p_4}<2$. Then, it easily follows that 
\[ 
\displaystyle\sum_{j=1}^\ell n_j - 3\ell = \begin{cases}
    \textrm{$-2$ or $-3$}, & \textrm{if $q_3=1$}, \\
    \textrm{$-4$ or $-5$}, & \textrm{if $q_3=2$}, \\
    \textrm{$-7$ or $-8$}, & \textrm{if $q_3=4$}.
\end{cases}
\]
Now suppose $p_4/q_4=[m_1,\dots,m_t]$. By combining this result with Lemma \ref{lem:3.3} below, we obtain \[ 
\displaystyle\sum_{i=1}^t m_i - 3t = \begin{cases}
    \textrm{0 or 1}, & \textrm{if $q_3=1$}, \\
    \textrm{2 or 3}, & \textrm{if $q_3=2$}, \\
    \textrm{5 or 6}, & \textrm{if $q_3=4$},
\end{cases}
\] 
and \begin{equation}\label{eq:c_and_KS^2}
    K_S^2=c-\frac{q_4+q_4^{-1}-2}{p_4},~\textrm{where }c=\begin{cases}
    \frac{20}{15}, & \textrm{if $\sum_i m_i -3t=0$ (and $q_3=1$)}, \\
    \frac{5}{15}, & \textrm{if $\sum_i m_i -3t=1$ (and $q_3=1$)}, \\
    \frac{26}{15}, & \textrm{if $\sum_i m_i -3t=2$ (and $q_3=2$)}, \\
    \frac{11}{15}, & \textrm{if $\sum_i m_i -3t=3$ (and $q_3=2$)}, \\
    \frac{17}{15}, & \textrm{if $\sum_i m_i -3t=5$ (and $q_3=4$)}, \\
    \frac{2}{15}, & \textrm{if $\sum_i m_i -3t=6$ (and $q_3=4$)}.
    \end{cases}
\end{equation}

\begin{lemma}[{\cite[Lemma 2.6]{Lisca-2007}}]\label{lem:3.3} For relatively prime integers $p>q>0$, suppose that \[
\frac{p}{p-q}=[n_1,\dots,n_\ell] ~~\textrm{and}\ \ ~~\frac{p}{q}=[m_1,\dots,m_t].
\]
If $\sum_{j=1}^\ell n_j =3\ell -k$, then $\sum_{i=1}^t m_i =3t+k-2$.    
\end{lemma}

\begin{remark} For a singularity of type $\{(2,1),(3,2),(5,q_3),(p_4,q_4)\}$, consider the plumbing diagram corresponding to $X:=X(2,1)\amalg X(3,2)\amalg X(5,q_3)\amalg X(p_4,q_4)$. Then we define $I(X)$ as the sum of the negative weights minus three times the length. From the above results, it follows that \[
I(X)= \begin{cases}
    \textrm{$-1$ or 0}, & \textrm{if $q_3=1$}, \\
    \textrm{$-2$ or $-1$}, & \textrm{if $q_3=2$ or 4}.
\end{cases}
\]
Also, the negative definite lattice $Q_X$, a direct sum of four linear lattices, must be embedded into the standard negative definite lattice $-\Bbb Z^N$ of codimension one. This concept aligns closely with Lisca's work \cite{Lisca-2007}, which classifies codimension \textit{zero} embeddings of linear lattices - specifically, the intersection forms of $X(p,q)$ into $-\Bbb Z^N$ on the assumption that $I(X(p,q))<0$. (For a more general treatment of codimension zero embeddings of direct sums of linear lattices, see also \cite{Lisca-2007-2}.) As implicitly shown in \cite[Section 7]{Lisca-2007}, if $I(X(p,q))<0$ and there exists a codimension zero embedding  $Q_{X(p,q)}\hookrightarrow -\Bbb Z^{b_2(X(p,q))}$, (considering $q^{-1}$ instead of $q$ if necessary) there exists a sequence  $(p(k),q(k))$ $(k=0,1,\dots)$ with $(p(0),q(0))=(p,q)$ such that $Q_{X(p(k),q(k))}$ embeds into the standard negative definite lattice as a codimension zero sublattice for each $k$. Moreover, it satisfies 
\begin{equation}\label{eq:family_of_(p_4,q_4)}
    \frac{p(k)}{q(k)}=\left[a_1,\dots,a_\ell, [2]^k,c_1,\dots,c_s,m+k, b_1,\dots,b_t\right],
\end{equation}
where $a_i, b_i,c_i$'s and $m$ are constants independent of $k$. 
\end{remark}

The forthcoming lemma, along with the subsequent remark, demonstrates that for a family of $(p_4,q_4)$'s characterized by form  (\ref{eq:family_of_(p_4,q_4)}), the value of $K_S^2$ of the type $\{(2,1),(3,2),(5,q_3),(p_4,q_4)\}$ approaches a rational number as $k$ tends to infinity. This limiting rational number is dependent solely on the constants $a_i$'s and the $b_i$'s.

\begin{lemma}\label{lem:3.4} 
Let $p>q>0$ be relatively prime integers such that 
\[
\frac{p}{q}=\left[a_1,\dots,a_\ell, [2]^k,c_1,\dots,c_s,m+k, b_1,\dots,b_t\right],
\]
where $a_i, b_i,c_i$'s and $m$ are constants $\geq 2$, $a_\ell>2$, and $\ell,s,t, k\geq 0$. Then the integers $p, q, q^{-1}$ are of the form $p=d_1k^2+e_1k+f_1$, $q=d_2k^2+e_2k+f_2$, and $q^{-1}=d_3k^2+e_3k+f_3$, with \[
\frac{d_1}{d_2}=\begin{cases}
    \left[a_1,\dots,a_{\ell-1},a_\ell-1\right], & \textup{if $\ell\geq 1$}, \\ 1, & \textup{if $\ell=0$}, \end{cases} \qquad and \qquad \frac{d_1}{d_3}=\begin{cases} [b_t,\dots,b_1], & \textup{if $t\geq 1$}, \\ \infty ~(d_3=0), & \textup{if $t=0$}.        
    \end{cases} \]    
\end{lemma}

\begin{proof} First note that \[ 
\left[[2]^k,x\right]=[\underbrace{2,\dots,2}_{k},x]=2-\frac{1}{\cdots-\displaystyle\frac{1}{2-\displaystyle\frac{1}{x}}}=\frac{(k+1)x-k}{kx-(k-1)}, ~~k=0,1,\dots;~x\in \Bbb R-\{0\},
\]
which can be established by induction on $k$. Since it is easily shown that $[c_1,\dots,c_s,m+k,b_1,\dots,b_t]$ is of the form $(ak+b)/(a'k+b')$ with $a>a'\geq 0$, we obtain \[
\left[[2]^k,c_1,\dots,c_s,m+k,b_1,\dots,b_t\right]=\frac{(a-a')k^2+(\textrm{lower degree terms})}{(a-a')k^2+(\textrm{lower degree terms})}.
\]
This shows that $d_1=d_2\neq 0$ if $\ell=0$. The assertion that $d_1/d_2=[a_1,\dots,a_{\ell-1},a_\ell -1]$ if $\ell \geq 1$ now easily follows from induction on $\ell$.

Next, we note that \[
\frac{p}{q^{-1}}=\left[b_t,\dots,b_1,m+k,c_s,\dots,c_1,[2]^k,a_\ell,\dots,a_1\right],
\]
and that $\left[c_s,\dots,c_1,[2]^k,a_\ell,\dots,a_1\right]$ is of the form $(a''k+b'')/(a'''k+b''')$ with $a''\geq a'''>0$. Thus, \[ \left[m+k,c_s,\dots,c_1,[2]^k,a_\ell,\dots,a_1\right]=\frac{a''k^2+(\textrm{lower degree terms})}{a''k+b''}.
\]
This shows that $d_1>d_3=0$ if $t=0$. The assertion that $d_1/d_3=[b_t,\dots,b_1]$ if $t\geq 1$ then follows from induction on $t$. \end{proof}

\begin{remark} Suppose that a family of $(p_4,q_4)$ is given as in Lemma \ref{lem:3.4}. Assuming that $p_4$ is relatively prime to 30, for $S$ with a singular type $\{(2,1),(3,2),(5,q_3),(p_4,q_4)\}$, we have \[
K_S^2=c-\frac{(d_2k^2+e_2k+f_2)+(d_3k^2+e_3k+f_3)-2}{d_1k^2+e_1k+f_1} \to c-\frac{d_2+d_3}{d_1} ~\textrm{as}~k\to \infty
\]
where $c$ is determined as in Equation (\ref{eq:c_and_KS^2}). Therefore, if we choose the $a_i$'s and $b_i$'s so that \[c-\frac{d_2+d_3}{d_1} \in \left(0,\displaystyle\frac{1}{10}\right),\] then we obtain an infinite family of singularity types $\{(2,1), (3,2), (5,q_3), (p_4,q_4)\}$ satisfying the oBMY inequality (note that $3e_{\textrm{orb}}(S)=1/10+3/p_4$).

\end{remark}

\begin{proposition}\label{lem:continued_fraction_length} The singularity type $\{(2,1),(3,q_2),(5,q_3),(p_4,q_4)\}$ \textup{(}with $(p_4,30)=1$\textup{)} cannot be realized by a rational homology $\mathbb{CP}^2$ if \[
\frac{p_4}{q_4}=[m_1,\dots,m_t]
\]
with $t\leq 3$. 
\end{proposition}
\begin{proof} As remarked above, we must have $q_2=2$ and \[ 
\displaystyle\sum_{i=1}^t m_i - 3t = \begin{cases}
    \textrm{$0$ or $1$}, & \textrm{if $q_3=1$}, \\
    \textrm{$2$ or $3$}, & \textrm{if $q_3=2$}, \\
    \textrm{$5$ or $6$}, & \textrm{if $q_3=4$}.
\end{cases}
\] 
In the case $q_3=1$, the possible choices of $[m_1,\dots,m_t]$, up to permutation of the $m_i$'s, are the followings: \begin{align*}
    [m_1,\dots,m_t] =&  [3], [4], \\
    &  [2,4], [3,3], [3,4], \\
    &  [2,2,5], [2,2,6], [2,3,4], [2,3,5], [2,4,4], [3,3,3], [3,3,4].
    \end{align*}
In the case $q_3=2$, we have: \begin{align*}
     [m_1,\dots,m_t] =& [5], [6], \\
     & [2,6], [2,7], [3,5], [3,6], [4,4], [4,5], \\
     & [2,2,7], [2,2,8], [2,3,6], [2,3,7],\dots, [3,4,4], [3,4,5], [4,4,4],
\end{align*}    
and, in the case $q_3=4$, we have: \begin{align*}
    [m_1,\dots,m_t] = & [8],[9], \\
    & [2,9], [2,10],[3,8],[3,9],[4,7],[4,8],[5,6],[5,7],[6,6] ,\\
    & [2,2,10], [2,2,11], [2,3,9],[2,3,10], \dots, [4,5,5],[4,5,6],[5,5,5].
\end{align*}

Now, consider $X=X(2,1)\natural X(3,2)\natural X(5,q_3)\natural X(p_4,q_4)$. In the case that $q_3=1$, a straightforward calculation (as in the proofs of Theorem \ref{thm:3.1} and Theorem \ref{thm:(2,3,11,13)}) reveals that the intersection form $Q_X$ of $X$ cannot be embedded into $-\mathbb{Z}^{n+1}$, where $n=b_2(X)$, except in the following cases (under the identification $[m_1,\dots,m_t]$ with $[m_t,\dots,m_1]$): 
\[
    [m_1,\dots,m_t]=[4], [2,2,6], [2,5,2], [2,5,3].
\]
In these instances, the values of $p_4$ are $4,16,16,\text{ and }25$, respectively. In each case, $p_4$ is not relatively prime to 30.

Next, consider the case that $q_3=2$. In this case, the intersection form $Q_X$ does not embed into $-\mathbb{Z}^{b_2(X)+1}$, except when $[m_1,\dots,m_t]=[4,3,4]$ and $[5,2,5]$. However, in both of these cases, we find that $p_4=40$, which is not relatively prime to 30.

Finally, in the case of $q_3=4$, $Q_X$ embeds in $-\mathbb{Z}^{b_2(X)+1}$ only for $[m_1,\dots,m_t]=[6,3,6]$. In this specific case, $p_4$ equals 96, which, similarly, is not relatively prime to 30.
\end{proof}

For the Hirzebruch-Jung continued fraction $\mathbf{x}=[x_1,\dots,x_n]$, we define a term \textit{length} of $\mathbf{x}$ as $n$. The following proposition demonstrates that a singularity type $\{(2,1),(3,q_2),(5,q_3),(p_4,q_4)\}$ fails to meet Donaldson's condition (Corollary \ref{cor:2.6}) in instances that the Hirzebruch-Jung continued fraction of $p_4/q_4$ contains sufficiently many $2$'s.

\begin{proposition}\label{lem:Donaldson_and_2's} For a singularity type $\{(2,1),(3,q_2),(5,q_3),(p_4,q_4)\}$ \textup{}with $(p_4,30)=1$\textup{}, suppose that the Hirzebruch-Jung continued fraction of $p_4/q_4$ is expressed in the form
\[
    \frac{p_4}{q_4}=\left[\mathbf{x}_0,[2]^{i_1},\mathbf{x}_1,[2]^{i_2},\mathbf{x}_2,\dots,\mathbf{x}_{r-1},[2]^{i_r},\mathbf{x}_r\right],
\]
where $i_j\geq 1$ for $j=1,\dots,r$, $\textup{length}(\mathbf{x}_i)\geq 0$ for $i=0,r$, and $\textup{length}(\mathbf{x}_i)\geq 1$ for $i=1,\dots,r-1$. Further suppose that one of the following holds: \begin{enumerate}[label=\normalfont(\arabic*)]
    \item $q_3=1$ and \[
\sum_{i=0}^r \textup{length}(\mathbf{x}_i) <r.
    \]
    \item $q_3=2$ and \[
\sum_{i=0}^r \textup{length}(\mathbf{x}_i) <r+1.
    \]
    \item $q_3=4$ and \[
\sum_{i=0}^r \textup{length}(\mathbf{x}_i) <r+2.
    \]
\end{enumerate}
Then the singularity type $\{(2,1),(3,q_2),(5,q_3),(p_4,q_4)\}$ cannot be realized by a rational homology $\mathbb{CP}^2$.    
\end{proposition}
\begin{proof} 
Since the proofs of the three cases are similar, we focus on the first case.

Further, since the length of a continued fraction of $p_4/q_4$ is 
\[
    t:=\sum_{j=1}^r i_j+ \sum_{i=0}^r \textrm{length}(\mathbf{x}_i),
\]
we can assume that $t\geq 4$ by Proposition \ref{lem:continued_fraction_length}. 

Consider $X=X(2,1)\natural X(3,2)\natural X(5,1)\natural X(p_4,q_4)$ and suppose that there exists an embedding $\iota:Q_X\hookrightarrow -\mathbb{Z}^{t+5}$. Let $u_1,u_2,u_3\in H_2(X;\mathbb Z)$ be the homology classes of the $(-2)$-spheres of the plumbing graph of $X(2,1)\natural X(3,2)$. Additionally, for each $j=1,\dots,r$, let $v_{j,1},\dots,v_{j,i_j}\in H_2(X;\mathbb Z)$ correspond to the homology classes of the $(-2)$-spheres of the plumbing graph of $X(p_4,q_4)$ associated with the $[2]^{i_j}$ sections. Refer to Figure \ref{fig:lem3.9}.

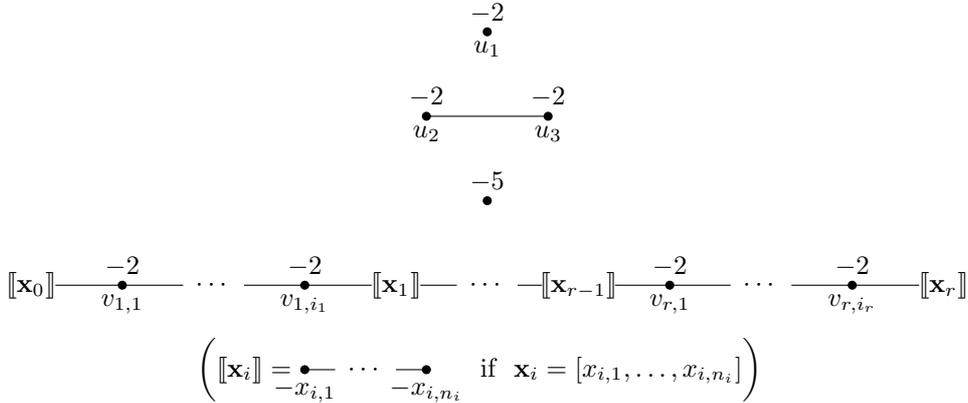
\begin{figure}[!th]
\centering
\begin{tikzpicture}[scale=0.8]
\draw (0,0) node[circle, fill, inner sep=1.2pt, black]{};
\draw (-1,-1.4) node[circle, fill, inner sep=1.2pt, black]{};
\draw (1,-1.4) node[circle, fill, inner sep=1.2pt, black]{};
\draw (0,-2.8) node[circle, fill, inner sep=1.2pt, black]{};
\draw (0,-4.2) node{$\cdots$};
\draw (-4.5,-4.2) node{$\cdots$};
\draw (4.5,-4.2) node{$\cdots$};
\draw (-6,-4.2) node[circle, fill, inner sep=1.2pt, black]{};
\draw (-3,-4.2) node[circle, fill, inner sep=1.2pt, black]{};
\draw (3,-4.2) node[circle, fill, inner sep=1.2pt, black]{};
\draw (6,-4.2) node[circle, fill, inner sep=1.2pt, black]{};
\draw (-7.5,-4.2) node{$\llbracket \mathbf{x}_0 \rrbracket$};
\draw (-1.5,-4.2) node{$\llbracket \mathbf{x}_1 \rrbracket$};
\draw (1.5,-4.2) node{$\llbracket \mathbf{x}_{r-1} \rrbracket$};
\draw (7.5,-4.2) node{$\llbracket \mathbf{x}_r \rrbracket$};

\draw (0,0) node[above]{$-2$};
\draw (-1,-1.4) node[above]{$-2$};
\draw (1,-1.4) node[above]{$-2$};
\draw (0,-2.8) node[above]{$-5$};
\draw (-3,-4.2) node[above]{$-2$};
\draw (-6,-4.2) node[above]{$-2$};
\draw (3,-4.2) node[above]{$-2$};
\draw (6,-4.2) node[above]{$-2$};

\draw (-1,-1.4)--(1,-1.4)   (-1.1,-4.2)--(-0.5, -4.2) (0.5,-4.2)--(0.9,-4.2) (-4,-4.2)--(-1.9,-4.2) (2.1,-4.2)--(4,-4.2) (5,-4.2)--(7.1,-4.2) (-7.1,-4.2)--(-5,-4.2);

\draw (0,0) node[below]{$u_1$};
\draw (-1,-1.4) node[below]{$u_2$};
\draw (1,-1.4) node[below]{$u_3$};
\draw (-6,-4.2) node[below]{$v_{1,1}$};
\draw (-3,-4.2) node[below]{$v_{1,i_1}$};
\draw (3,-4.2) node[below]{$v_{r,1}$};
\draw (6,-4.2) node[below]{$v_{r,i_r}$};

\draw (-3,-5.6) node[circle, fill, inner sep=1.2pt, black]{};
\draw (-1,-5.6) node[circle, fill, inner sep=1.2pt, black]{};
\draw (-4,-5.6) node{$\bigg( \llbracket \mathbf{x}_i \rrbracket = $};
\draw (-3,-5.6) node[below]{$-x_{i,1}$};
\draw (-1,-5.6) node[below]{$-x_{i,n_i}$};
\draw (-2,-5.6) node{$\cdots$};
\draw (-3,-5.6)--(-2.5,-5.6) (-1.5,-5.6)--(-1,-5.6)  ;
\draw (2.2,-5.6) node{if $~~\mathbf{x}_i=[x_{i,1},\dots,x_{i,n_i}]\bigg)$};
\end{tikzpicture}
\caption{The plumbing graph of $X(2,1)\natural X(3,2)\natural X(5,1)\natural X(p_4,q_4).$}
\label{fig:lem3.9}
\end{figure}

We first consider the classes $v_{1,1},\dots,v_{1,i_1}$. Since $\iota(v_{1,1})^2=-2$, it follows that $\iota(v_{1,1})=\pm e_i\pm e_j$ for some $1\leq i<j\leq t+5$. We can then assume $\iota(v_{1,1})=e_1-e_2$ after a basis change. For $\iota(v_{1,2})=\pm e_i\pm e_j$ with $1\leq i<j\leq t+5$, and given $\iota(v_{1,1})\cdot \iota(v_{1,2})=1$, we must have $|\{1,2\}\cap \{i,j\}|=1$, leading us to assume $i=2$ and $\iota(v_{1,2})=e_2-e_3$. If $i_1>2$, then write $\iota(v_{1,3})=\pm e_i\pm e_j$ with $1\leq i<j\leq t+5$. Considering $v_{1,1}\cdot v_{1,3}=0$ and $v_{1,2}\cdot v_{1,3}=1$, we conclude that either $\iota(v_{1,3})=\pm(e_1+e_2)$ or $i=3$. However, if $\iota(v_{1,3})=\pm(e_1+e_2)$, then, as $t\geq 4$, there exists a class $v\in H_2(X;\mathbb{Z})$ such that $v\cdot v_{1,3}=1$ while $v\cdot v_{1,1}=0$, leading to a contradiction: $0=\iota(v_{1,1})\cdot \iota(v)=(e_1-e_2)\cdot \iota(v) \equiv  \pm(e_1+e_2)\cdot \iota(v)=\iota(v_{1,3})\cdot \iota(v)=1 \mod 2$. Thus, we must have $i=3$ and we can assume $\iota(v_{1,3})=e_3-e_4$. Continuing this process allows us to assume that $\iota(v_{1,k})=e_k-e_{k+1}$ for each $k=1,\dots,i_1$.

Next, for the classes $v_{2,1},\dots,v_{2,i_2}$, we write $\iota(v_{2,1})=\pm e_i\pm e_j$ for $1\leq i<j\leq t+5$. Given that $v_{2,1}\cdot v_{1,k}=0$ for each $k=1,\dots,i_1$, it follows that $i>i_1+1$. (In the case that $i_1=1$, we need an assumption that $t\geq 4$ as in the preceding paragraph.) Therefore, we may assume $\iota(v_{2,1})=e_{i_1+2}-e_{i_1+3}$, and similarly, that $\iota(v_{2,k})=e_{i_1+k+1}-e_{i_1+k+2}$ for each $k=1,\dots,i_2$.

Repeating this procedure, we obtain 
\[
    \iota(v_{j,k})=e_{\left(\sum_{\ell=1}^{j-1}i_\ell
    \right)+ k + j-1}- e_{\left(\sum_{\ell=1}^{j-1}i_\ell \right)+ k + j}
\]
for each $j=1,\dots,r$ and $k=1,\dots,i_j$. (If $\sum_{\ell=1}^r i_\ell + r > t+5$, i.e., $r-5>\sum_{i=0}^r\textrm{length}(\mathbf{x}_i)$, then $Q_{X(p_4,q_4)}$ does not embed in $-\mathbb{Z}^{t+5}$, resulting in a contradiction.)

Finally, for the classes $u_1, u_2, u_3$, we write $\iota(u_k)=\pm e_{i'k}\pm e{j'_k}$ for $k=1,2,3$, with $1\leq i'_k <j'_k\leq t+5$. Then the following conditions must hold: 
\begin{align*}
    \{i'_1,j'_1\} \cap \{i'_k, j'_k\} & =\emptyset~~(k=2,3), \\
    |\{i'_2,j'_2\}\cap \{i'_3,j'_3\}| & =1,\\
    \{i'_k, j'_k\} \cap \bigg\{1,2,\dots, \sum_{\ell=1}^r i_\ell + r\bigg\} & = \emptyset ~~(k=1,2,3).    
\end{align*}
It follows that the set $\{1,2,\dots,\sum_{\ell=1}^r i_\ell+r, i_1',j_1',i_2',j_2',i_3',j_3'\}$ must have order $\sum_{\ell=1}^r i_\ell +r +5$ (after removing duplicates). However, this is impossible as $\sum_{\ell=1}^r i_\ell +r +5>t+5$, given that $r>\sum_{i=0}^r \textrm{length}(\mathbf{x}_i)$. This contradiction indicates that an embedding $\iota:Q_X\hookrightarrow -\mathbb{Z}^{t+5}$ is not feasible. Hence we conclude that the singularity type $\{(2,1),(3,2),(5,1),(p_4,q_4)\}$ cannot be realized by a rational homology $\mathbb{CP}^2$ by Corollary \ref{cor:2.6}.
\end{proof}

Now, a methodology for identifying infinite families of types stated in Theorem \ref{thm:infinite} becomes evident. The technique involves utilizing Lemma \ref{lem:3.4} to derive a family of pairs $(p_4,q_4)$ that meet the orbifold BMY inequality. If these pairs feature a sufficient number of $2$'s in the Hirzebruch-Jung continued fraction, then Proposition \ref{lem:Donaldson_and_2's}, which follows from Donaldson's theorem, is applicable to effectively obstruct the entirety of this family. The following explicit example illustrates this approach. It will become clear that similar families can be readily constructed using the techniques presented here.

\begin{example} Consider a singularity of type $\{(2,1),(3,2),(5,4),(p_4,q_4)\}$, where \[
\frac{p_4}{q_4}=\left[[2]^k,2+k,9\right], ~~k=0,1,\dots.
\]
By applying Lemma \ref{lem:3.4}, we obtain ${d_1}/{d_2}=1$ and ${d_1}/{d_3}=9$. 
A direct computation yields 
\[
    p_4=9k^2+17k+17, ~~q_4=9k^2+8k+9,~~\textrm{and}~~ q_4^{-1}=k^2+2k+2.
\]
Since $2k+(2+k)+9 - 3(k+2)=5$, we find $c={17}/{15}$ in this case, leading to \[c-\frac{d_2+d_3}{d_1}=\frac{1}{45}\in (0,\frac{1}{10}).\] In fact, 
\[
    K_S^2=\frac{17}{15}-\frac{(9k^2+8k+9)+(k^2+2k+2)-2}{9k^2+17k+17}=\frac{3k^2+139k+154}{15(9k^2+17k+17)}.
\]
Therefore, assuming that $(p_4,30)=1$, which holds if and only if $k\neq 2\mod3$, the type 
\[
    \{(2,1),(3,2),(5,4),(9k^2+17k+17,9k^2+8k+9)\} 
\]
satisfies the oBMY inequality for each $k\geq 12$ with $k\neq 2\mod 3$. However, according to Proposition \ref{lem:Donaldson_and_2's}, the intersection form of \[
X(2,1)\natural X(3,2)\natural X(5,4)\natural X(9k^2+17k+17,9k^2+8k+9)
\]
cannot embed into $-\mathbb{Z}^{k+10}$ for all $k$. 

This argument extends to the singularity type $\{(2,1),(3,2),(5,4),(p_4,q_4)\}$ where 
\[
    \frac{p_4}{q_4}=\left[[2]^{k+b_1-9}, 2+k, b_1\right], ~~9\leq b_1\leq 30,~ k\gg 1,
\]
and more generally, 
\[
    \frac{p_4}{q_4}=\left[[2]^{k+b_n-n-8},2+k, [2]^{n-1}, b_n\right], ~~ n\geq 1,~ 9\leq b_n\leq 30, ~ k\gg 1.
\]
\end{example}

For the sake of completeness, we also illustrate families in the cases where $q_3=1$ and $q_3=2$, which provide more examples of Theorem \ref{thm:infinite}:

\begin{example} Consider a singularity of type $\{(2,1),(3,2),(5,2),(p_4,q_4)\}$, where \[
\frac{p_4}{q_4}=\left[[2]^{k-5},2+k,2,2\right], ~~k\geq5.
\]
Then we have \[
p_4=3k^2-11k-1, ~~q_4=3k^2-14k-2,~~ q_4^{-1}=2k^2-7k-2,
\]
and, in particular, $(p_4,30)=1$ if and only if $k\neq 1\mod 3$. For all such $k$, the singularity type cannot be realized by a rational homology $\mathbb{CP}^2$ according to Proposition \ref{lem:Donaldson_and_2's}. However, by Lemma \ref{lem:3.4}, the singularity type satisfies the oBMY inequality for infinitely many $k$. In fact, we have \[
K_S^2= \frac{26}{15}-\frac{q_4+q_4^{-1}-2}{p_4}=\frac{3k^2+29k+64}{15(3k^2-11k-1)},
\]
and the type satisfies the oBMY inequality for all $k\geq 32$ with $k\neq 1\mod 3$.
\end{example}

\begin{example} Consider a singularity of type $\{(2,1),(3,2),(5,1),(p_4,q_4)\}$, where 
\[
    \frac{p_4}{q_4}=\left[[2]^{k-1},2+k,2,4\right], ~~k=1,2,\dots.
\]
Then we have 
\[
    p_4=7k^2+3k+7, ~~q_4=7k^2-4k+4,~~ q_4^{-1}=2k^2+k+2,
\]
and particularly, $(p_4,30)=1$ for all $k$. We also find 
\[
    K_S^2= \frac{20}{15}-\frac{q_4+q_4^{-1}-2}{p_4}=\frac{k^2+21k+16}{3(7k^2+3k+7)},
\]
indicating that the type satisfies the oBMY inequality for all $k\geq 19$.

Although we cannot directly apply Proposition \ref{lem:Donaldson_and_2's} for all $k$, assuming $k$ is sufficiently large, we can demonstrate that the given singularity type does not satisfy Donaldson's condition, as follows: Let $X=X(2,1)\natural X(3,2)\natural X(5,1)\natural X(7k^2+3k+7,7k^2-4k+4)$ and suppose there exists an embedding $\iota:Q_X\hookrightarrow -\mathbb{Z}^{k+7}$. Let $u_1,u_2,u_3,u_4\in H_2(X;\mathbb Z)$ be the homology classes corresponding to the 2-spheres of the plumbing graph of $X(2,1)\natural X(3,2)\natural X(5,1)$, and let $v_1,\dots,v_{k+2}\in H_2(X;\mathbb Z)$ be the homology classes corresponding to the 2-spheres of the plumbing graph of $X(p_4,q_4)$. Refer to Figure \ref{fig:example_(p_4,q_4)_q_3=1}. Proceeding as in the proof of Proposition \ref{lem:Donaldson_and_2's}, we can assume the following: \begin{align*}
    \iota(v_i)&=e_i-e_{i+1} ~~(i=1,\dots,k-1, k+1), \\
    \iota(u_1) & = e_{k+3}-e_{k+4}, \\
    \iota(u_i) & = e_{k+3+i}-e_{k+4+i} ~~(i=2,3).
\end{align*}

Next, consider the vector $\iota(v_{k+2})$. Since $v_{k+2}^2=-4$, $\iota(v_{k+2})$ can either be of the form $\pm 2 e_i$ or $\pm e_{i_1}\pm e_{i_2}\pm e_{i_3}\pm e_{i_4}$. In the former case, $\iota(v_{k+2})\cdot \iota (v_{k+1})=1$ cannot be satisfied, so we consider the latter form and write $\iota(v_{k+2})=\pm e_{i_1}\pm e_{i_2}\pm e_{i_3}\pm e_{i_4}$ with $i_1<\cdots<i_4$. Given that $1=\iota(v_{k+2})\cdot \iota (v_{k+1})=\iota(v_{k+2})\cdot (e_{k+1}-e_{k+2})$, we deduce that $\{i_1,\dots,i_4\}\cap \{k+1,k+2\}$ must contain exactly one element, which we can assume to be $k+2$. Assuming $k\geq 4$, it follows that $\{i_1,\dots,i_4\}\cap \{1,\dots,k\}=\emptyset$, leading to $\{i_1,\dots,i_4\}\subset \{k+2,\dots,k+7\}$. By considering that $v_{k+2}\cdot u_i=0$ for $i=1,2,3$, we find that $\{i_1,\dots,i_4\}=\{k+2,k+5,k+6,k+7\}$, and we may assume that (cf. Definition \ref{def:unique_embedding}) \[
\iota(v_{k+2})=e_{k+2}+e_{k+5}+e_{k+6}+e_{k+7}.
\]

For the vector $\iota(u_4)$, written as $\iota(u_4)=\sum_{j=1}^{k+7}a_je_j$, the relation that $u_4\cdot v_i=0$ for $i=1,\dots,k-1$ implies $a_1=\cdots=a_k$. Assuming $k\geq 6$, it follows that $a_1=\cdots=a_k=0$. Hence, $\iota(u_4)=\sum_{j=k+1}^{k+7} a_je_j$, and the coefficients must satisfy 
\[
    a_{k+1}=a_{k+2},~a_{k+3}=a_{k+4},~a_{k+5}=a_{k+6}=a_{k+7},~\textrm{and}~a_{k+2}+a_{k+5}+a_{k+6}+a_{k+7}=0.
\]
However, this leads to the conclusion that $\iota(u_4)^2\neq -5$, a contradiction. This implies that such an embedding $\iota$ cannot exist for each $k\geq 6$ (in fact, such an $\iota$ does not exist for all $k\geq 1$ except for $k=3$), and that the singularity type \[
\{(2,1),(3,2),(5,1),(7k^2+3k+7,7k^2-4k+4)\}
\]
cannot be realized by a rational homology $\mathbb{CP}^2$ for each $k\geq 1$.
\end{example}

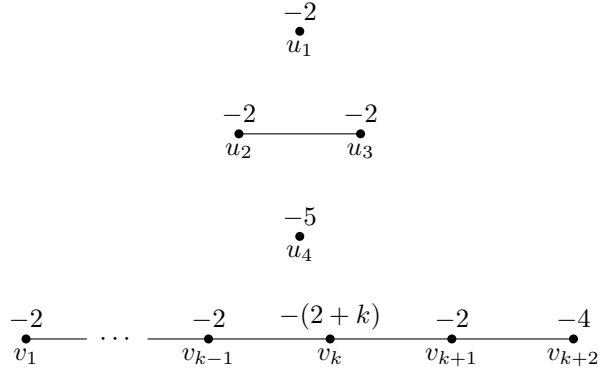
\begin{figure}[!th]
\centering
\begin{tikzpicture}[scale=0.8]
\draw (0,0) node[circle, fill, inner sep=1.2pt, black]{};
\draw (-1,-1.7) node[circle, fill, inner sep=1.2pt, black]{};
\draw (1,-1.7) node[circle, fill, inner sep=1.2pt, black]{};
\draw (0,-3.4) node[circle, fill, inner sep=1.2pt, black]{};
\draw (-4.5,-5.1) node[circle, fill, inner sep=1.2pt, black]{};
\draw (-1.5,-5.1) node[circle, fill, inner sep=1.2pt, black]{};
\draw (0.5,-5.1) node[circle, fill, inner sep=1.2pt, black]{};
\draw (2.5,-5.1) node[circle, fill, inner sep=1.2pt, black]{};
\draw (4.5,-5.1) node[circle, fill, inner sep=1.2pt, black]{};

\draw (0,0) node[above]{$-2$};
\draw (-1,-1.7) node[above]{$-2$};
\draw (1,-1.7) node[above]{$-2$};
\draw (0,-3.4) node[above]{$-5$};
\draw (-4.5,-5.1) node[above]{$-2$};
\draw (-1.5,-5.1) node[above]{$-2$};
\draw (0.5,-5.1) node[above]{$-(2+k)$};
\draw (2.5,-5.1) node[above]{$-2$};
\draw (4.5,-5.1) node[above]{$-4$};

\draw (-1,-1.7)--(1,-1.7) (-4.5,-5.1)--(-3.5,-5.1) (-2.5,-5.1)--(4.5,-5.1);

\draw (-3,-5.1) node{$\cdots$};
\draw (0,0) node[below]{$u_1$};
\draw (-1,-1.7) node[below]{$u_2$};
\draw (1,-1.7) node[below]{$u_3$};
\draw (0,-3.4) node[below]{$u_4$};
\draw (-4.5,-5.1) node[below]{$v_1$};
\draw (-1.5,-5.1) node[below]{$v_{k-1}$};
\draw (0.5,-5.1) node[below]{$v_k$};
\draw (2.5,-5.1) node[below]{$v_{k+1}$};
\draw (4.5,-5.1) node[below]{$v_{k+2}$};
\end{tikzpicture}
\caption{The plumbing graph of $X(2,1)\natural X(3,2)\natural X(5,1)\natural X(7k^2+3k+7, 7k^2-4k+4).$}
\label{fig:example_(p_4,q_4)_q_3=1}
\end{figure}

\section{A mysterious infinite family}\label{subsec:mysterious_infinite_family}
Although many conditions from the study of smooth 4-manifolds proved effective in ruling out most types with four singularities, unfortunately, we were able to construct an infinite family of types that cannot be obstructed from realizing a simply-connected rational homology $\mathbb{CP}^2$ by any conditions discussed in this paper. Consider the following family of singularity types:
    \[\left\{(2,1),(3,2),(5,1),(204s^2+732s+649,12(17s^2+44s+20))\right\}, ~~s=0,1,2,\dots\]
We claim that such a type of singularities satisfies all the conditions given in Section \ref{sec:conditions} for infinitely many $s$. More precisely, let $s_n=5k_n+3$, where  \[
k_n=-\frac{7}{2}+\frac{1}{60}\left( \left(285+71\sqrt{15} \right)\cdot \left(31-8\sqrt{15}\right)^n + \left(285-71\sqrt{15} \right)\cdot \left(31+8\sqrt{15}\right)^n \right).
\]for $n=0,1,2,\dots$. For example, $s_0=33$, $s_1=38$, $s_2=3193$, and $s_3=198798$, $\dots$.
\begin{proposition}\label{prop:mysterious} 
The family of singularity types 
\[\left\{(2,1),(3,2),(5,1),(204s_n^2+732s_n+649,12(17s_n^2+44s_n+20))\right\}, ~~n=0,1,\dots\]
satisfy all the conditions mentioned in Section~\ref{sec:conditions}.    
\end{proposition}

\begin{proof}
(1) Note that $204s^2+732s+649$ is relatively prime to 30 unless $s\equiv 1\mod 5$. So, from now on, we assume that $s\not\equiv 1\mod 5$. 

(2) For $S$ with a singularity of type $\{(2,1),(3,2),(5,1),(204s^2+732s+649, 12(17s^2+44s+20))\}$, we have 
    \[K_S^2=\frac{2(12s^2+348s+653)}{3(204s^2+732s+649)}>0 ~~\textrm{for all $s$},\]
and 
    \[3e_{\textrm{orb}}(S)-K_S^2= \frac{372s^2-4764s-11023}{30(204s^2+732s+649)}>0~~\textrm{if $s\geq 15$}.\]
Therefore, $S$ satisfies the orbifold BMY inequality for $s\geq 15$. 

(3) We have 
\[
    D=20(12s^2+348s+653)=2^2\cdot 5\cdot (12s^2+348s+653).
\]
For $D$ to be a square number, we must have $12s^2+348s+653=0 \mod 5$, and this holds only if $s=3\mod 5$. Writing $s=5k+3$, we get 
\[
    D=10^2 (60k^2+420k+361),
\]
and $D$ is a square number if and only if $60k^2+420k+361$ is a square number. Also, $60k^2+420k+361$ is a square number for infinitely many positive integers $k$. In fact, if we let \[
k_n=-\frac{7}{2}+\frac{1}{60}\left( \left(285+71\sqrt{15} \right)\cdot \left(31-8\sqrt{15}\right)^n + \left(285-71\sqrt{15} \right)\cdot \left(31+8\sqrt{15}\right)^n \right)
\]
and \[
a_n=\frac{1}{2}\left(\left(71+19\sqrt{15} \right)\cdot \left(31-8\sqrt{15}\right)^n + \left(71-19\sqrt{15} \right)\cdot \left(31+8\sqrt{15}\right)^n  \right)
\]
for $n=0,1,2,\dots$, then we have \[
60k_n^2+420k_n+361=a_n^2
\]
for all $n$. By using the binomial theorem and induction, it can be checked that $k_n$ and $a_n$ are positive integers for all $n$. Therefore, the squareness of $D$ condition (Proposition \ref{prop:D_square}) holds for infinitely many values of $s$.

(4) Next, we check the linking form condition (Proposition \ref{prop:linking_form}). We need to demonstrate that $-q= 3876s^2+20028s+24601 \mod 30(204s^2+732s+649)$ is a quadratic residue. Through a direct calculation, it can be shown that $-q\equiv x^2 \mod 30(204s^2+732s+649)$, where \[
x=\begin{cases}
    1632s^2+5754s+4979, & \textrm{if $s=0\mod 5$}, \\
    2856s^2+10146s+8873, & \textrm{if $s=2$ or $4 \mod 5$}, \\
    408s^2+1362s+1085, & \textrm{if $s=3\mod 5$}.
\end{cases}
\]

(5) Now we consider the condition from Donaldson's theorem. We have \[
\frac{204s^2+732s+649}{12(17s^2+44s+20)}=\left[[2]^{s},3,4,2,3+s,3,2,4\right],
\]
and the intersection form $Q_X$ of $X:=X(2,1) \natural X(3,2)\natural X(5,1) \natural X(204s^2+732s+649, 12(17s^2+44s+20))$ embeds into $-\Bbb Z^{s+12}$ via the map given in Figure \ref{fig:mysterious_family}. Tubing 2-spheres corresponding to the red vertices gives a smoothly embedded characteristic 2-sphere satisfying Theorem \ref{thm:2.7}. Moreover, assuming $s\neq 1\mod 5$, the orthogonal complement of $Q_X$ in $-\Bbb Z^{s+12}$ is generated by $\sum_{i=1}^{s+12}a_ie_i$ where $a_{12}=a_{13}=\cdots=a_{12+s}$ and  \begin{align*}
    (& a_1, \dots,a_{12}) \\
     & = (-(6s+9),-(6s+9),10,10,-10, 30,30,4(6s+9),-4(6s+9),-6(6s+9),10(6s+9),60).
\end{align*}
Since \begin{align*}
    (6 & s+9)^2 \cdot 2 + 10^2 \cdot 3+30^2\cdot 2+4^2(6s+9)^2\cdot 2+6^2(6s+9)^2+10^2(6s+9)^2+60^2\cdot (s+1) \\
    & = 2\cdot 3\cdot 5\cdot (204s^2+732s+649),
\end{align*}
the condition in Proposition \ref{prop:2.8} is also satisfied.

\begin{figure}[!th]
\centering
\begin{tikzpicture}[scale=1]
\draw (0,0) node[circle, fill, inner sep=1.2pt, black]{};
\draw (-1,-1.4) node[circle, fill, inner sep=1.2pt, black]{};
\draw (1,-1.4) node[circle, fill, inner sep=1.2pt, black]{};
\draw (0,-2.8) node[circle, fill, inner sep=1.2pt, red]{};
\draw (0,-4.2) node[circle, fill, inner sep=1.2pt, black]{};
\draw (-5,-4.2) node[circle, fill, inner sep=1.2pt, black]{};
\draw (-2.5,-4.2) node[circle, fill, inner sep=1.2pt, red]{};
\draw (3,-4.2) node[circle, fill, inner sep=1.2pt, red]{};
\draw (6,-4.2) node[circle, fill, inner sep=1.2pt, black]{};
\draw (-6,-5.6) node[circle, fill, inner sep=1.2pt, black]{};
\draw (7,-5.6) node[circle, fill, inner sep=1.2pt, black]{};
\draw (-6,-7) node[circle, fill, inner sep=1.2pt, black]{};
\draw (7,-7) node[circle, fill, inner sep=1.2pt, black]{};

\draw (-6,-6.2) node{$\vdots$};

\draw (0,0) node[above]{$-2$};
\draw (-1,-1.4) node[above]{$-2$};
\draw (1,-1.4) node[above]{$-2$};
\draw (0,-2.8) node[above]{$-5$};
\draw (0,-4.2) node[above]{$-2$};
\draw (-5,-4.2) node[above]{$-3$};
\draw (-2.5,-4.2) node[above]{$-4$};
\draw (3,-4.2) node[above]{$e_7-e_{11}+(e_{12}+\cdots+e_{12+s})$};
\draw (6,-4.2) node[above]{$-3$};
\draw (-6,-5.6) node[left]{$-2~$};
\draw (7,-5.6) node[left]{$-2~$};
\draw (-6,-7) node[left]{$-2~$};
\draw (7.1,-7) node[right]{$-4$};

\draw (-1,-1.4)--(1,-1.4) (-5,-4.2)--(-2.56,-4.2) (-2.44,-4.2)--(2.94,-4.2) (3.06,-4.2)--(6,-4.2) (7,-5.6)--(7,-7) (6.36,-4.7)--(7,-5.6) (-6,-5.6)--(-6,-6.03) (-6,-6.58)--(-6,-7) (-5.36, -4.7)--(-6,-5.6);

\draw (0,0) node[below]{$e_1-e_2$};
\draw (-1,-1.4) node[below]{$-e_3+e_4$};
\draw (1,-1.4) node[below]{$e_3+e_5$};
\draw (0,-2.8) node[below]{$-e_1-e_2-e_8+e_9-e_{10}$};
\draw (0,-4.2) node[below]{$e_6-e_7$};
\draw (-5,-4.2) node[below]{$e_6+e_7-e_{12}$};
\draw (-2.5,-4.2) node[below]{$e_3+e_4-e_5-e_6$};
\draw (3,-4.2) node[below]{$-(3+s)$};
\draw (6,-4.2) node[below]{$-e_8+e_{10}+e_{11}$};
\draw (-6,-5.6) node[right]{$~e_{12}-e_{13}$};
\draw (7.1,-5.6) node[right]{$e_8+e_9$};
\draw (-6,-7) node[right]{$~e_{11+s}-e_{12+s}$};
\draw (6.9,-7) node[left]{$e_1+e_2-e_8-e_{10}$};
\draw [decorate,decoration={brace,amplitude=3pt},xshift=0pt]
	(-3.5,-5.5) -- (-3.5,-7.1) node [black,midway,xshift=8pt,yshift=0pt] 
	{$s$};
\end{tikzpicture}
\caption{An embedding of $Q_{X(2,1)}\oplus Q_{X(3,2)}\oplus Q_{X(5,1)}\oplus Q_{X(204s^2+732s+649,12(17s^2+44s+20))}$ into $ -\mathbb{Z}^{s+12}$.}
\label{fig:mysterious_family}
\end{figure}
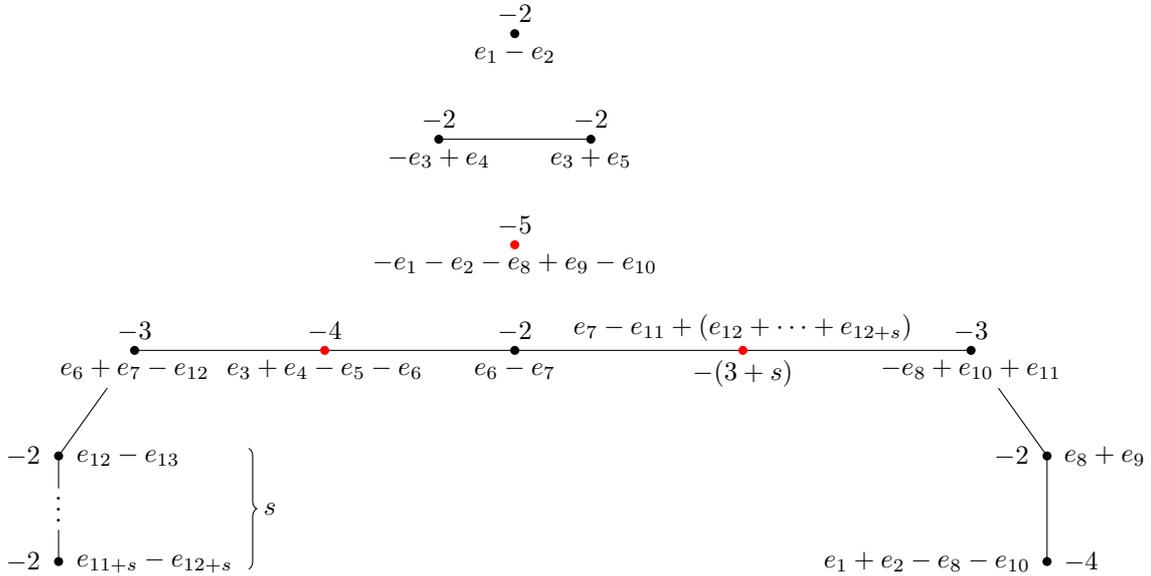 

(6) Finally, we check the spin $d$-invariant condition (Corollary \ref{cor:2.10}). Note that $60s^2+216s+192$ is the inverse of $12(17s^2+44s+20)$ modulo $204s^2+732s+649$. A straightforward application of the recursive formula (\ref{eq:d-invariant}) yields that the $d$-invariant of $L(204s^2+732s+649, 60s^2+216s+192)$ with the spin structure $132s^2+474s+420$ is zero for all $s$. Since $d(L(2,1),1)=\frac{1}{4}$, $d(L(3,2),2)=\frac{1}{2}$, and $d(L(5,1),0)=-1$, this implies that $L(2,1)\# L(3,2)\# L(5,1)\# L(204s^2+732s+649, 60s^2+216s+192)$ has a spin structure with $d=-\displaystyle\frac{1}{4}$.
\end{proof}

\bibliography{references}{}
\bibliographystyle{alpha}
\end{document}